\documentclass[preprint,times,12pt]{elsarticle}

\usepackage{amssymb}
\usepackage{amsthm}
\usepackage{amsmath,amsfonts,amssymb,euscript}
\usepackage{dsfont}
\usepackage{picinpar,rotating,amsbsy,calc,ifthen}
\usepackage{longtable,makeidx}
\usepackage{graphics}
\usepackage{graphicx}
\usepackage{epsfig}
\usepackage{epstopdf}
\usepackage{float}
\usepackage{subfigure}
\usepackage{color}
\usepackage{placeins}
\usepackage{tikz}
\usetikzlibrary{arrows}
\usepackage{soul}
\usepackage{enumerate}

\newtheorem{newthm}{Proposition}

\newtheorem{newlem}{Lemma}
\newtheorem{newass}{Assumption}

\newtheorem{newrem}{Remark}
\newtheorem{exmp}{Example}[section]

\journal{Physica D}

\begin{document}

\begin{frontmatter}



\title{Dynamics of Transcription-Translation Networks
}

\address[label1]{Department of Mathematics and Statistics, University of Victoria,
PO Box 1700, STN CSC, Victoria, BC, V8W 2Y2, Canada}

\author[label1]{D. Hudson}
\ead{drhh@uvic.ca}

\author[label1]{R. Edwards\corref{cor1}}
\ead{edwards@uvic.ca}

\cortext[cor1]{Corresponding author at: Tel.: +250-721-7453; Fax: +250-721-8962}

\date{\today}

\begin{abstract}
A theory for qualitative models of gene regulatory networks has been developed over several decades, generally considering transcription factors to regulate directly the expression of other transcription factors, without any intermediate variables. Here we explore a class of models that explicitly includes both transcription and translation, keeping track of both mRNA and protein concentrations. We mainly deal with transcription regulation functions that are steep sigmoids or step functions, as is often done in protein-only models, though translation is governed by a linear term. We extend many aspects of the protein-only theory to this new context, including properties of fixed points, description of trajectories by mappings between switching points, qualitative analysis via a state-transition diagram, and a result on periodic orbits for negative feedback loops. We find that while singular behaviour in switching domains is largely avoided, non-uniqueness of solutions can still occur in the step-function limit.

\end{abstract}


\end{frontmatter}


	\section{Introduction}\label{sec:intro}
	
	Models of gene regulatory networks often omit many biochemical details, partly because parameters in specific systems are often not well known, but also because it is argued that qualitative behaviour, at least, will be similar in simplified models. For example, a good deal of work on developing general-purpose tools for analysis of the behaviour of gene networks has modeled only concentrations of proteins that act as transcription factors, as if these proteins directly regulated production of other proteins. We know that this is not really the case --- proteins regulate the transcription of mRNA's that in turn produce proteins by translation. There may also be post-translational modifications to a protein before it is effective as a regulator. It is often argued that the time scales of the dynamics of mRNA and protein are vastly different, so that it is not unreasonable to consider mRNA dynamics to be infinitely fast, so that only the protein variables need be retained in a model. 
	
	Typically, mRNA decay rates are significantly higher than those of proteins, or equivalently, protein half life tends to be longer. However, these time scales may not always be so different and the range of ratios of these decay rates is highly variable across genes and organisms (see, for example,~\cite{Bernstein2002,Bionumbers,Gedeon2012,Mosteller1980,Paetkau2006}). In previous work, it has been shown that behaviour of transcription-translation networks and the correspondong protein-only networks can differ qualitatively, even in some cases when time scales of the two types of variable are very different (but not infinitely different)~\cite{Edwards2014,Gedeon2012}.
	
	This observation makes it desirable to develop a method of analysis for trans-cription-translation networks. One can still use the simplifying assumption that the regulation (promotion or repression) becomes effective sharply at a particular threshold, so that the regulatory effect as a function of protein concentration is a very steep sigmoid, or even infinitely steep. A start to an analysis of such systems was made in a previous paper~\cite{Edwards2014}, but the focus there was on a comparison of the transcription-translation system to its protein-only counterpart. One of the main advantages of a transcription-translation model, from the point of view of analysis, is that there is no self-input of any variable as a regulator of its own production. If, biochemically, a gene is autoregulating, the process is now modeled as a feedback loop between the gene's mRNA transcript, and the corresponding translated protein. Thus, the difficulties that arise in protein-only networks with `black walls' (trajectories approach a threshold hyperplane from both sides), `white walls' (trajectories move away from a threshold hyperplane on both sides) and sliding in walls (trajectories confined to a threshold hyperplane for a nonzero time interval, while moving in other variables) no longer arise. There are still sensitive behaviours at intersections of walls that require careful analysis, but the problems of singular flow seem to be avoided in typical solution trajectories.
	
	On the other hand, even in the case of infinitely steep switching, maps between threshold transitions are no longer as easy to calculate, and contrary to the protein-only case, trajectories can reverse direction without crossing a threshold. These issues are explored more fully here, and we show that, in fact, the direction reversal leads to particular trajectories that graze a threshold hyperplane tangentially, leading once again to non-uniqueness of some solutions in the infinitely steep switching case (Section~\ref{subsec:nonunique}). It is possible, however, to divide phase space up into regions (which will here be called {\it pseudo-state domains}) in such a way that flows are logically captured by a directed graph in which nodes represent regions, in a similar way to what is done for protein-only networks, even though here, only half the variables have thresholds (Section~\ref{sec:graphrep}). Negative feedback loops still correspond to cycles on such a state-transition graph, and with appropriate parameter values, these have a corresponding unique locally stable periodic solution that is also qualitatively stable with respect to the (adjacent) boxes through which it passes. 
	
	We investigate a number of other properties of trajectories of the transcription-translation model, in a way that parallels the theory for protein-only networks. For example, we show that a fixed point in a regular domain (we use this term also in the limit of infinitely-steep switching, where it becomes a region of phase space bounded by threshold hyperplanes) is still necessarily asymptotically stable, but not globally with respect to that regular domain, unlike the protein-only case (Section~\ref{subsec:fixedpoints}). In Sections~\ref{subsec:localsol} and \ref{sec:globaldynam}, we determine how to calculate the map from one threshold transition to the next, though in practice this requires numerically finding a root of a transcendental equation in most situations (this was partially done in~\cite{Edwards2014}, but not every case was covered there). We finish with a summary of what has been achieved and discussion of implications.
	
	\section{The Protein Only Model}\label{sec:prelim}
		In this work, we are interested in qualitative descriptions of gene regulatory networks. A class of simplified models, proposed by Glass~\cite{Glass1977}, and elaborated by others (for example \cite{EI2013,Farcot2006,pk2005}), describe $n$-gene networks by an $n$-dimensional system of differential equations with either a step function or a sigmoidal interaction term. Using the notation of Plahte and Kj{\o}glum, the equations are
			\begin{equation}
				\dot{y}_{i}=F_{i}(Z)-\beta_{i}y_{i}, \ \ i=1,\dots,n\,,
			\label{Model1}
			\end{equation}
where $\beta_{i}>0$ is constant and $Z=(Z_{11},\dots,Z_{np_{n}})$ is a vector of sigmoid functions $Z_{ij}=\mathcal{S}(y_{i},\theta_{ij},q)$ satisfying a number of conditions laid out in their paper~\cite{pk2005}. Here $y_{i}$ denotes the concentration of the $i^{th}$ protein, $\theta_{ij}$ is the switching threshold of $Z_{ij}$, $j\in \{0,1,\dots,p_{i}\}$, and $q$ is a steepness parameter. The functions $F_{i}(Z)\geq 0$ are multilinear polynomials, i.e., affine with respect to each $Z_{ij}$. Inherently, production rates are bounded, so there exist positive constants $\bar{F}_{i}$ such that $0 \leq F_{i}(Z) \leq \bar{F}_{i}$ for each $i\in \{1, \dots, n\}$. We define $\theta_{i0}=0$ and $\theta_{i,p_{i}+1}=y_{i,\textup{max}}:= \frac{\bar{F}_{i}}{\beta_{i}}$. 
		
		As in~\cite{EI2013}, we take $\mathcal{S}(y_{i},\theta_{ij},q)$ to be the Hill function $H(y_{i},\theta_{ij},q)$,
			\begin{equation}
				H(y_{i},\theta_{ij},q)=\dfrac{y_{i}^{\frac{1}{q}}}{y_{i}^{\frac{1}{q}}+\theta_{ij}^{\frac{1}{q}}}.
			\label{eqn;hillfunc}
			\end{equation}
		Note that 
			\begin{equation*}
				\lim_{q \rightarrow 0}H(y_{i},\theta_{ij},q)= \Bigg \{
					\begin{array}{lcl}
						0 & \textup{if} & y_{i}<\theta_{ij} \\
						1 & \textup{if} & y_{i}>\theta_{ij} 
					\end{array}
					.
			\end{equation*}
		Since for each gene $i$ we assign one equation, we refer to (\ref{Model1}) as Model 1.
		
		In the limit as $q \rightarrow0$, phase space can be divided into boxes, 
		\[\mathcal{B}_{j_{1}\dots,j_{n}}=\prod_{i=1}^{n}(\theta_{ij_{i}},\theta_{i,j_{i}+1})\,,\quad j_{i}\in \{0,1,\dots,p_{i}\}\,,\] 
		separated by threshold hyperplanes. Flow in each box is directed towards a focal point $\Phi_i=\frac{F_i(Z)}{\beta_i}$, for the value of the binary vector $Z$ appropriate to the box ($Z_{ij}=0$ if $y_i<\theta_{ij}$ and $Z_{ij}=1$ if $y_i>\theta_{ij}$). If a fixed point lies inside its own box, then no switching occurs and the trajectory converges asymptotically to the focal point, which is then an asymptotically stable fixed point (this straightforward result has been observed many times; see, for example,~\cite{GK1973,Glass1975,gs2002,Wittmann2009}). Otherwise, mappings from threshold to threshold can be calculated. One can apply these maps iteratively to get a long term mapping that one can use to give conditions for existence and stability of periodic solutions. See, for instance,~\cite{Edwards2000} or~\cite{Farcot2006}. 
		 
	\section{The Transcription-Translation Model}\label{sec:modelintro}
	\hspace{.25in}
		A $2n$-dimensional model explicitly describing both the transcription and translation steps has been proposed in~\cite{Edwards2014} and~\cite{Gedeon2012}:
			\begin{equation}
				\begin{array}{lcl} 
					\dot{x}_{i} & = & F_{i}(Z)-\beta_{i}x_{i} \\
					\dot{y}_{i} & = & \kappa_{i}x_{i}-\gamma_{i}y_{i}
				\end{array}
			\ \ \  i=1,\dots, n\,,
			\label{def;Model2}
			\end{equation}
		which we refer to as Model 2 henceforth. In Model 2, $x_{i}$ represents the concentration of the $i^{th}$ mRNA and $y_{i}$ represents the concentration of the protein product for gene $i$. We take $Z=(Z_{1j},\dots, Z_{np_{n}})$, where each $Z_{ij}$ is as before. Again, we take $\mathcal{S}(y_{i},\theta_{ij},q)=H(y_{i},\theta_{ij},q)$ to be the Hill function defined in (\ref{eqn;hillfunc}). We take each $F_{i}$ and $\beta_{i}$ to be defined as before, and add that $\gamma_{i}>0$, and $\kappa_{i}>0$. All the examples we present will deal with the limit case $q\rightarrow 0$, but the main results will be shown for both $q\rightarrow 0$ and for $q>0$.
		
		We first note that since $\dot{y}_{i}$ is independent of $Z_{ij}$, all threshold hyperplanes $y_{i}=\theta_{ij}$ are \textit{transparent}, i.e. solution trajectories pass through them. 
		
		The threshold hyperplanes $y_{i}=\theta_{ij}$ divide $\mathbb{R}^{2n}$ into regions that we call regular domains. To be more precise, we adapt some notation from~\cite{Farcot2006}: let $\mathbb{N}_{p_{i}}=\{0,1,\dots, p_{i}\}$ and let $\mathcal{H}=\prod_{i=1}^{n} \mathbb{N}_{p_{i}}$. For consistency, we declare that $\theta_{i,0}=0$ and $\theta_{i,p_{i}+1}=y_{i,\textup{max}}$. It follows that $y_{i}$ has $p_{i}$ thresholds. Let $h\in \mathcal{H}$. We define a \textit{Regular Domain}, $\mathcal{D}_{h}$, in the limit $q\to 0$, to be
			\begin{equation}
				\mathcal{D}_{h}=\mathcal{D}_{h_{1},\dots,h_{n}}=\mathbb{R}_{+}^{n}\times\prod_{i=1}^{n}(\theta_{i,h_{i}},\theta_{i,h_{i}+1}),\ \ h_{i}\in \mathbb{N}_{p_{i}}.
			\end{equation} 
			Note that for $q>0$, the intervals $(\theta_{i,h_{i}},\theta_{i,h_{i}+1})$ have to be replaced by $(\theta_{i,h_{i}}+\delta(q),\theta_{i,h_{i}+1}-\delta(q))$, where $\delta(q)\to 0$ as $q\to 0$, and the switching regions have thickness that vanishes as $q\to 0$.
		Inside a regular domain none of the $y_{i}$ are at threshold value. For $0<q\ll 1$, the sigmoid vector $Z$ can be approximated by a binary vector $B$, and it converges to $B$ as $q\rightarrow0$. Thus, inside regular domains in the limit as $q\rightarrow0$, each $F_{i}(Z)$ is a constant, $\alpha_i$ (which implicitly still depends on $Z$, of course). Consequently, in a regular domain $\mathcal{D}_{h}$, Equations~(\ref{def;Model2}) can be solved uniquely in the limit as $q\rightarrow 0$ and these solutions will hold until one of the $y_{i}$ hits a threshold. Solutions must be directed towards a focal point, 
		\begin{equation} \Phi=(x^*,y^*)=(x_1^*,\ldots,x_n^*,y_1^*,\ldots,y_n^*)\mbox{ where }(x_i^*,y_i^*)=\left(\frac{\alpha_i}{\beta_i},\frac{\kappa_i \alpha_i}{\gamma_i \beta_i}\right)\,,\label{eq:fp}\end{equation}
		monotonically in each $x_i$, but not necessarily in each $y_i$.
		 		
	\section{Local Dynamics in a Regular Domain}\label{sec:localdynam}
	\hspace{.25in}
		In this section we talk about local dynamics in regular domains, and compare with local dynamics in Model 1. For what follows, we make the following assumption:
			\begin{newass}
				No focal point, $\varPhi=(x^*,y^*)$, from (\ref{eq:fp}), for any binary vector $Z$, lies on a threshold, i.e. $\frac{\kappa_i \alpha_i}{\gamma_i \beta_i}=\frac{\kappa_i}{\gamma_i \beta_i}F_i(Z) \neq \theta_{i,h_{i}}$ for any $i>0$ and $h_{i}\in \mathbb{N}_{p_{i}}$.
			\label{ass; ass1}
			\end{newass}

		\subsection{Solutions in a Regular Domain}\label{subsec:localsol}
		We begin by solving (\ref{def;Model2}) in a regular domain, $\mathcal{D}_{h}$ for $q\rightarrow0$. In $\mathcal{D}_{h}$, (\ref{def;Model2}) takes the form  
			\begin{equation}
				\begin{array}{lcl} 
					\dot{x}_{i} & = & \alpha_{i}-\beta_{i}x_{i} \\
					\dot{y}_{i} & = & \kappa_{i}x_{i}-\gamma_{i}y_{i}
				\end{array}
			\ \ \  i=1,\dots, n.
			\label{eqn;reg}
			\end{equation}
		Let $N=\{1,\dots,n\}$. Let $J=\{i\in N\ |\ \beta_{i}=\gamma_{i}\}$ and let $I=N\setminus J=\{i\in N\ |\ \beta_{i}\neq \gamma_{i}\}$. We can then write (\ref{eqn;reg}) as 
			\begin{equation}
				\begin{array}{ll}
					\begin{array}{lcl}
						\dot{x}_{i} & = & \alpha_{i}-\beta_{i}x_{i} \\
						\dot{y}_{i} & = & \kappa_{i}x_{i}-\gamma_{i}y_{i}
					\end{array} & i \in I\,, \\
					\begin{array}{lcl}
						\dot{x}_{j} & = & \alpha_{j}-\gamma_{j}x_{j} \\
						\dot{y}_{j} & = & \kappa_{j}x_{j}-\gamma_{j}y_{j}
					\end{array} & j \in J\,,
				\end{array}
			\label{eqn;generalreg}
			\end{equation} 
		with initial conditions $x_{i}(0)=x_{i}^{(h)}$, $x_{j}(0)=x_{j}^{(h)}$, $y_{i}(0)=y_{i}^{(h)}$, and $y_{j}(0)=y_{j}^{(h)}$, and $(x_{i}^{(h)},y_{i}^{(h)},x_{j}^{(h)},y_{j}^{(h)})\in \mathcal{D}_{h}$. Solving the first and third equation in (\ref{eqn;generalreg}) gives
			\begin{equation}
				\begin{array}{ll}
					x_{i}(t)=\alpha_{i}/\beta_{i}+(x_{i}^{(h)}-\alpha_{i}/\beta_{i})e^{-\beta_{i}t}, & i \in I \\ 
					x_{j}(t)=\alpha_{j}/\gamma_{j}+(x_{j}^{(h)}-\alpha_{j}/\gamma_{j})e^{-\gamma_{j}t}, & j \in J\,.
				\end{array}
			\label{eqn;regxsol}
			\end{equation}
		We insert these equations into the second and fourth equations from (\ref{eqn;generalreg}) and rearrange to get
			\begin{equation*}
				\begin{array}{l}
					\dot{y}_{i}+\gamma_{i}y_{i}=\kappa_{i}\alpha_{i}/\beta_{i}+\kappa_{i}(x_{i}^{(h)}-\alpha_{i}/\beta_{i})e^{-\beta_{i}t}\\
					\dot{y}_{j}+\gamma_{j}y_{j}=\kappa_{j}\alpha_{j}/\gamma_{j}+\kappa_{j}(x_{j}^{(h)}-\alpha_{j}/\gamma_{j})e^{-\gamma_{j}t}\,.
				\end{array}
			\end{equation*}
		Solving gives
			\begin{equation}
				\begin{array}{l}
					y_{i}(t)=\dfrac{\kappa_{i}\alpha_{i}}{\beta_{i}\gamma_{i}}+\dfrac{\kappa_{i}(x_{i}^{(h)}-\alpha_{i}/\beta_{i})}{\gamma_{i}-\beta_{i}}e^{-\beta_{i}t}+\left(y_{i}^{(h)}-\dfrac{\kappa_{i}\alpha_{i}}{\beta_{i}\gamma_{i}}-\dfrac{\kappa_{i}(x_{i}^{(h)}-\alpha_{i}/\beta_{i})}{\gamma_{i}-\beta_{i}}\right)e^{-\gamma_{i}t} \\
					\\
					y_{j}(t)=\dfrac{\kappa_{j}\alpha_{j}}{\gamma_{j}^{2}}+\left(\kappa_{j}x_{j}^{(h)}-\dfrac{\kappa_{j}\alpha_{j}}{\gamma_{j}}\right)te^{-\gamma_{j}t}+\left(y_{j}^{(h)}-\dfrac{\kappa_{j}\alpha_{j}}{\gamma_{j}^{2}}\right)e^{-\gamma_{j}t}\,.
				\end{array}
			\label{eqn;regysol}	
			\end{equation}
		It follows that in the regular domain $\mathcal{D}_{h}$, the solution to (\ref{eqn;generalreg}) is given by (\ref{eqn;regxsol}) and (\ref{eqn;regysol}). Since all the threshold hyperplanes in Model 2 (equations~(\ref{def;Model2})) are transparent, the solution can be continuously extended from one regular domain to an adjacent one by concatenating trajectory segments calculated between threshold intersections.  
		
		It is clear that in each variable, $x_i$ simply approaches $\frac{\alpha_i}{\beta_i}$ exponentially. In the $(x_i,y_i)$-plane, the $y_i$ null cline is the line $x_i=\frac{\gamma_i}{\kappa_i}y_i$. For larger $x_i$ values (or smaller $y_i$ values), $y_i$ is increasing, while for smaller $x_i$ values (or larger $y_i$ values), $y_i$ is decreasing. Fixed points fall on the intersection of the null clines $x_i=\frac{\alpha_i}{\beta_i}$ and $x_i=\frac{\gamma_i}{\kappa_i}y_i$ when this falls inside $(\theta_{h_i},\theta_{h_{i+1}})$.
		
		We note that these solutions lack monotonicity in $y_{i}$. This feature gives rise to dynamics not seen in Model 1, including solutions that graze a threshold tangentially. 
			
	\subsection{Fixed Points in Regular Domains}\label{subsec:fixedpoints}
		We start with a result that carries over from Model 1, which we rely on throughout.
		\begin{newthm}
			Let $(x^{*},y^{*})=(x^{*}_{1},\dots,x^{*}_{n},y_{1}^{*},\dots,y_{n}^{*})$, $x^*_{i}=\frac{\alpha_{i}}{\beta_{i}}$, $y^*_{i}=\frac{\kappa_{i}\alpha_{i}}{\gamma_{i}\beta_i}$, be a fixed point of (\ref{eqn;reg}) in a regular domain $\mathcal{D}_{h}$ of Model 2 in the limit $q\to 0$ (i.e., the focal point $\Phi=(x^{*},y^{*})$ for $\mathcal{D}_{h}$ lies inside $\mathcal{D}_{h}$). Then, $(x^{*},y^{*})$ is a locally asymptotically stable node. The same is true with $q$ sufficiently small for the corresponding, slightly shifted, fixed point.
		\label{stableprop}
		\end{newthm}
		\begin{proof}
			In a regular domain with $q\rightarrow0$, $Z$ is a fixed binary vector, so we can write $F_{i}(Z)=\alpha_{i}$, and system (\ref{def;Model2}) becomes (\ref{eqn;reg}). 
			
			The Jacobian of (\ref{eqn;reg}) at $(x^{*},y^{*})$ is then
				\begin{equation*}J(x^{*},y^{*})= \left(
					\begin{array}{cccccccc}
						-\beta_{1} & 0 & \dots & 0 & 0 & 0 & \dots & 0 \\
						0 & -\beta_{2} & \dots & 0 & 0 & 0 & \dots & 0 \\
						\vdots & \ddots & \dots & \ddots & \ddots & \ddots & \dots & \vdots \\
						0 & 0 & \dots & -\beta_{n} & 0 & 0 & \dots & 0 \\
						\kappa_{1} & 0 & \dots & 0 & -\gamma_{1} & 0 & \dots & 0 \\
						0 & \kappa_{2} & \dots & 0 & 0 & -\gamma_{2} & \dots & 0 \\ 
						\vdots & \ddots & \dots & \ddots & \ddots & \ddots & \dots & \vdots \\
						0 & 0 & \dots & \kappa_{n} & 0 & 0 & \dots & -\gamma_{n} 
					\end{array}
					\right)
				\end{equation*}
			which clearly has eigenvalues $-\beta_{i}$ and $-\gamma_{i}$ for $i=1, \dots, n$. Since $\beta_{i}>0$ and $\gamma_{i}>0$ for each $i$, all the eigenvalues are negative, so $(x^{*},y^{*})$ is locally asymptotically stable.
				
			For $q>0$ it can be shown that 
				\begin{equation*}
					\lim_{q\rightarrow 0}\frac{dZ_{mj}}{d{y_{m}}}=0
				\end{equation*}
			for all $m$ and $j$. Thus, for $q>0$ sufficiently small, by continuity of eigenvalues with respect to matrix entries, the result still holds.
		\end{proof}
		
We emphasize the locality of this result, as it is shown below that it is possible to enter a regular domain $\mathcal{D}_{h}$ containing a fixed point, and not converge to it. This is impossible in Model 1,
where trajectories were monotone in each variable within a given regular domain. Accordingly, trajectories could not leave a regular domain via the threshold they entered. This is no longer true in Model 2. 
		
		As a result of the lack of monotonicity, there can be solutions that {\it graze} a threshold, {\em i.e.,} intersect it tangentially (see Figure~\ref{fig:prop2}). If $\frac{\alpha_{i}}{\beta_{i}}<\frac{\gamma_{i}}{\kappa_{i}}\theta_{i,h_{i}+1}$ then there exists a solution trajectory on a curve $\Gamma_{u_{i}}^{(h_{i})}$ in $\mathcal{D}_{h}$ that grazes $\theta_{i,h_{i}+1}$, and if $\frac{\alpha_{i}}{\beta_{i}}>\frac{\gamma_{i}}{\kappa_{i}}\theta_{i,h_{i}}$ then there exists a solution trajectory on a curve $\Gamma_{l_{i}}^{(h_{i})}$ in $\mathcal{D}_{h}$ that grazes $\theta_{i,h_{i}}$. 
Their existence has interesting implications, including bounding a trapping region and non-uniqueness of solutions in the limit $q \rightarrow 0$.
		
		We construct these curves in the limit $q\rightarrow 0$. Then, $\Gamma_{l_{i}}^{(h_{i})}$, when it exists, is defined to be the unique solution trajectory in phase space that reaches $(\frac{\gamma_{i}}{\kappa_{i}}\theta_{i,h_{i}},\theta_{i,h_{i}})$. We can find an initial point $(x_{i}^{(0)},\theta_{i,h_{i}+1})$ on the threshold $y_{i}=\theta_{i,h_{i}+1}$, when it exists, as follows. Using one of the equations in (\ref{eqn;regxsol}) (depending on whether or not $\beta_{i}=\gamma_{i}$), set $x_{i}(t)=\frac{\gamma_{i}}{\kappa_{i}}\theta_{i,h_{i}}$, then solve for $t$. Insert this $t$ into the appropriate equation in (\ref{eqn;regysol}), then set $y_{i}(t)=\theta_{i,h_{i}}$ and $y_{i}^{(0)}=\theta_{i,h_{i}+1}$. Finally solve for $x_{i}^{(0)}$. If this cannot be done explicitly, it can be done numerically. If $x_{i}^{(0)}<0$ by this calculation, we still define $\Gamma_{l_{i}}^{(h_{i})}$ in this way, even though  $x_{i}^{(0)}<0$ is not physically meaninful. Let $t^{*}$ be the time it takes for $x_{i}(t;x_{i}^{(0)})$ to get to $\frac{\gamma_{i}}{\kappa_{i}}\theta_{i,h_{i}}$, i.e. $x_{i}(t^{*};x_{i}^{(0)})=\frac{\gamma_{i}}{\kappa_{i}}\theta_{i,h_{i}}$. $\Gamma_{l_{i}}^{(h_{i})}$ is then given parametrically by 
				\begin{equation*}
					\Bigg\{
					\begin{array}{lcl}
						x_{i}(t;x_{i}^{(0)}) \\
						y_{i}(t;\theta_{i,h_{i}+1})
					\end{array}
					\textup{for} \ \ 0\leq t \leq t^{*}.
				\end{equation*}
		The curve $\Gamma_{u_{i}}^{(h_{i})}$ is defined and found analogously, replacing $\theta_{i,h_{i}}$ with $\theta_{i,h_{i}+1}$ and vice versa. 
			
The existence of fixed points in regular domains depends on the existence of these curves. 

\begin{newthm}
In any regular domain, $\mathcal{D}_{h}$, there exists a unique fixed point at $(x^*,y^*)=\left(\frac{\alpha_i}{\beta_i},\frac{\kappa_i}{\gamma_i}\frac{\alpha_i}{\beta_i}\right)$ if and only if $\Gamma_{l_{i}}^{(h_{i})}$ and $\Gamma_{u_{i}}^{(h_{i})}$ exist for all $i=1,\ldots,n$, i.e. if $\frac{\gamma_i}{\kappa_i}\theta_{i,h_i}<\frac{\alpha_i}{\beta_i}<\frac{\gamma_i}{\kappa_i}\theta_{i,h_i+1}$ for all $i=1,\ldots,n$.
\label{fixedpoint}\end{newthm}
\begin{proof}This is simply a consequence of the fact that for each $i$, the $x_i$ null cline ($x_i=\frac{\alpha_i}{\beta_i}$) and the $y_i$ null cline ($y_i=\frac{\kappa_i}{\gamma_i}x_i$) intersect in $\mathcal{D}_{h}$.
\end{proof}
			
Regions bounded by these curves form trapping regions. To define such regions, the cases where a $y_i$ variable is below $\theta_{i,1}$ or above $\theta_{i,p_i}$ must be handled separately. 

For each regular domain, $\mathcal{D}_{h}$, we define the $2n$-dimensional region:
\begin{equation*} 
    \mathcal{R}_{c}^{(h)}:=\displaystyle{\prod_{i=1}^{n}\mathcal{R}_{c_{i}}^{(h_{i})}},
\end{equation*}
where each $\mathcal{R}_{c_{i}}^{(h_{i})}$ is a two-dimensional region defined as follows:

				\textbf{Case 1:} If $0<h_i<p_i$, so that both $\Gamma_{u_{i}}^{(h_{i})}$ and $\Gamma_{l_{i}}^{(h_{i})}$ exist, 
					\begin{equation*}
						\mathcal{R}_{c_{i}}^{(h_{i})}:=\{(x_{i},y_{i})\in \mathbb{R}^{2}_{+}|\theta_{i,h_{i}}<y_{i}<\theta_{i,h_{i}+1},\ \max\{\Gamma_{l_{i}}^{(h_{i})},0\}<x_{i}<\Gamma_{u_{i}}^{(h_{i})} \}.
					\end{equation*}
					
				\textbf{Case 2:} If $h_i=p_i$, so that only $\Gamma_{l_{i}}^{(h_{i})}$ exists, 
					\begin{equation*}
						\mathcal{R}_{c_{i}}^{(h_{i})}:=\{(x_{i},y_{i})\in \mathbb{R}^{2}_{+}|\theta_{i,p_{i}}<y_{i},\ \Gamma_{l_{i}}^{(h_{i})}<x_{i} \}.
					\end{equation*}
					
				\textbf{Case 3:} If $h_i=0$, so that only $\Gamma_{u_{i}}^{(h_{i})}$ exists, 
 					\begin{equation*}
						\mathcal{R}_{c_{i}}^{(h_{i})}:=\{(x_{i},y_{i})\in \mathbb{R}^{2}_{+}|0<y_{i}<\theta_{i,1},\ 0<x_{i}<\Gamma_{u_{i}}^{(h_{i})} \}.
					\end{equation*} 

Note that it is allowed for the $x$-value of $\Gamma_{l_{i}}^{(h_{i})}$ to go below $0$, and one can find the intersection of $\Gamma_{l_{i}}^{(h_{i})}$ with the $y$-axis to describe the appropriate region in a similar way to the calculation of $x_i^{(0)}$ on the threshold $y_i=\theta_{i,h_i+1}$ above.

			\begin{newthm}
				Let $q \rightarrow 0$. Suppose there exists a fixed point in a regular domain $\mathcal{D}_{h}$. Then, for each $i$, if a trajectory enters $\mathcal{D}_{h}$ at $(x^{(0)},y^{(0)}) \in \mathcal{R}_{c}^{(h)}$
 then no other switching occurs and it converges to the fixed point. Conversely, $\mathcal{D}_{h}\setminus\mathcal{R}_{c}^{(h)}\neq\emptyset$ and if there exists at least one pair $(x_{i}^{(0)},y_{i}^{(0)})$ such that the trajectory enters $\mathcal{D}_{h}$ at $(x^{(0)},y^{(0)})$ with $(x_{i}^{(0)},y_{i}^{(0)})\in \mathcal{D}_{h_{i}}\setminus\mathcal{R}_{c_{i}}^{(h_{i})}$, then the trajectory will leave $\mathcal{D}_{h}$.
			\label{fixpointprop}
			\end{newthm}
			\begin{proof}
				Let $\mathcal{D}_{h}$ be a regular domain containing a fixed point
					\begin{equation*}
						(x^{*},y^{*})=(x^{*}_{1},\dots,x^{*}_{n},y^{*}_{1},\dots,y^{*}_{n}).
					\end{equation*}
				In $\mathcal{D}_{h}$ the flow is given by (\ref{eqn;regxsol}) and (\ref{eqn;regysol}), and we note that we can solve both equations in (\ref{eqn;regxsol}) for $t$ as a function of $x_{i}$ without difficulty.
								
				Since each pair $(x_{i},y_{i})$ is independent of $(x_{j},y_{j})$ for all $i \neq j$ in a regular domain, we can consider each pair by itself. Suppose a trajectory $\Pi$ enters a regular domain at a point $(x^{(0)},y^{(0)}) \in \mathcal{R}_{c}^{(h)}$. Then $(x_{i},y_{i})\in \mathcal{R}_{c_{i}}^{(h_{i})}$ for each $i=1,\dots,n$. By uniqueness, the direction of the flow across the boundaries, and the construction of $\Gamma_{u_{i}}^{(h_{i})}$ (resp. $\Gamma_{l_{i}}^{(h_{i})}$),  $\mathcal{R}_{c_{i}}^{(h_{i})}$ is an invariant region, so, by Proposition 1 and the Poincar\'e-Bendixson Theorem, $(x_{i},y_{i})$ converges to $(x_{i}^{*},y_{i}^{*})$. Since this is true for each $i$, we conclude that $\Pi$ converges to $(x^{*},y^{*})$.
							
				Conversely, since there is no {\em a priori} maximum value for any $x_{i}$, and $\kappa_{i}\neq 0$, $\gamma_{i}\neq 0$, it follows immediately that $\mathcal{D}_{h}\setminus\mathcal{R}_{c}^{(h)}\neq\emptyset$. Suppose that $(x_{i},y_{i})\notin \mathcal{R}_{c_{i}}^{(h_{i})}$ for some $i$. Then, again by uniqueness and the construction of $\Gamma_{u_{i}}^{(h_{i})}$ (resp. $\Gamma_{l_{i}}^{(h_{i})}$), one of these $y_{i}$ must cross a threshold, at which point $(x^{*},y^{*})$ ceases to be the relevant focal point and $\Pi$ flows towards a new focal point. 			
			\end{proof}
				\begin{figure}[!htb]
					\hspace*{-1.53cm}
					\includegraphics[scale=.4]{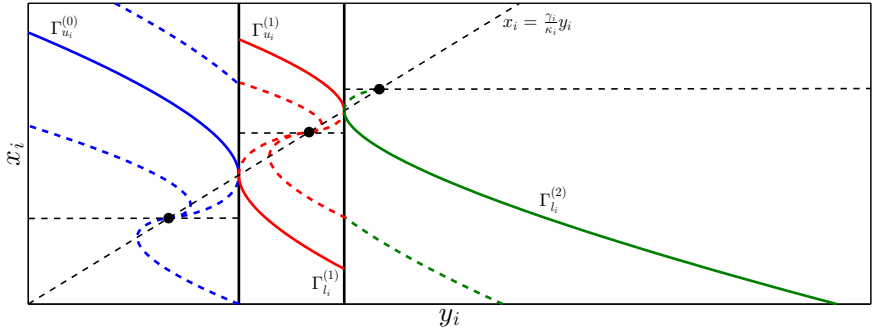}
					\caption{Phase plane for a single pair $(x_i,y_i)$. The solid vertical lines represent threshold values $\theta_{i,h_{i}}$ and $\theta_{i,h_{i}+1}$, respectively. The cases 1, 2 and 3 correspond to the region between the two thresholds, to the right of the right threshold, and to the left of the left threshold, respectively. In this example, we take the $x_{i}$ nullcline, $x_{i}=\frac{F_{i}(Z)}{\beta_{i}}$ in each region to be in $\mathcal{R}_{c_{i}}^{(h_{i})}$. The grazing trajectories, curves $\Gamma_{l_{i}}^{(h_{i})}$ and $\Gamma_{u_{i}}^{(h_{i})}$, are in bold, while other trajectories are shown as dotted curves. Null-clines are shown as dotted straight lines, and fixed points as solid circles.}
				\label{fig:prop2}
				\end{figure} 
					
		By analysis of the flow, we can conclude the existence of a sepratrix in each regular domain that corresponds to $\Gamma_{u_{i}}^{(h_{i})}$ or $\Gamma_{l_{i}}^{(h_{i})}$ in the case of $q>0$, but they are more difficult to calculate. However, if certain conditions are met (see section \ref{subsec:nonunique}) then these curves correspond to the stable manifolds of fixed points near the thresholds. 
		
		This result describes a loss of global stability with respect to the regular domains when we move from the Model 1 to Model 2, illustrated in Example~\ref{prop2ex}, below.
		
			\begin{newrem}
				Refer to Figure~\ref{fig:prop2}. Recall that the existence of $\Gamma_{l_{i}}^{(h_{i})}$ or $\Gamma_{u_{i}}^{(h_{i})}$ did not rely on the existence of a fixed point. Also, as a corollary of Proposition \ref{fixpointprop}, if $h_{i}<p_{i}$ and $\frac{\alpha_{i}}{\beta_{i}}< \frac{\gamma_{i}}{\kappa_{i}}\theta_{i,h_{i}}$ for some $\theta_{i,h_{i}}$, then there is an interval $\mathcal{R}_{i,h_{i}}$ of the threshold line $\theta_{i,h_{i}}$ in $\mathbb{R}_{+}^{2}$ such that if a trajectory $\Pi$ enters $\mathcal{D}_{h}$ on the threshold $\theta_{i,h_{i}}$ at $(\bar{x},\bar{y})$ with $\bar{x}_i \in \mathcal{R}_{i,h_{i}}$, then the next threshold $\Pi$ crosses cannot be $\theta_{i,h_{i}+1}$. This interval is given by
					\begin{equation*}
						\mathcal{R}_{i,h_{i}}=\bigg\{(x_{i},y_{i})\in\mathbb{R}^{2}_{+}\ \bigg|\ \frac{\gamma_{i}}{\kappa_{i}}\theta_{i,h_{i}}<x_{i}<\Gamma_{u_{i}}^{(h_{i})}, y_{i}=\theta_{i,h_{i}}\bigg \}.
					\end{equation*}	
				 Similarly, if $\frac{\alpha_{i}}{\beta_{i}}> \frac{\gamma_{i}}{\kappa_{i}}\theta_{i,h_{i}+1} $, there is an interval $\mathcal{R}_{i,h_{i}+1}$ of the threshold line $\theta_{i,h_{i}+1}$ such that if a trajectory $\Pi$ enters $\mathcal{D}_{h}$ on the threshold $\theta_{i,h_{i}+1}$ at $(\bar{x},\bar{y})$ with $\bar{x}_i \in \mathcal{R}_{i,h_{i}}$, then the next threshold $\Pi$ crosses will not be $\theta_{i,h_{i}}$. This region is given by
				 \begin{equation*}
					\mathcal{R}_{i,h_{i}+1}=\{(x_{i},y_{i})\in\mathbb{R}^{2}_{+}\ |\ \Gamma_{l_{i}}^{(h_{i})}<x_{i}<\frac{\gamma_{i}}{\kappa_{i}}\theta_{i,h_{i}+1}, y_{i}=\theta_{i,h_{i}} \}.
				 \end{equation*} 
			\end{newrem}
			
			\begin{exmp}
				Consider the 2-gene network in the framework of Model 1 
					\begin{equation}
						\begin{array}{ll}
							\dot{y}_{1}=Z_{11}+\frac{3}{2}(Z_{21}-Z_{22})-\frac{7}{8}x_{1} \\
							\dot{y}_{2}=Z_{11}+\frac{3}{2}Z_{12} - \frac{7}{8}x_{2} \\
						\end{array}
					\label{eqn;instbl}
					\end{equation}
with $\theta_{11}=1$, $\theta_{12}=2$, $\theta_{21}=\frac{1}{2}$, and $\theta_{22}=1$. We note that $(0,0)$ is a fixed point and any solution entering the box $\mathcal{B}_{00}:=[0,1]\times[0,\frac{1}{2}]$ will converge to it asymptotically. 
				
				We now expand the system to include both mRNA and protein as follows,
					\begin{equation}
						\begin{array}{l}
							\dot{x}_{1}=Z_{11}+\frac{3}{2}(Z_{21}-Z_{22})-\frac{7}{8}x_{1} \\
							\dot{y}_{1}=x_{1}-y_{1} \\
							\dot{x}_{2}=Z_{11}+\frac{3}{2}Z_{12} - \frac{7}{8}x_{2} \\
							\dot{y}_{2}=x_{2}-2y_{2},
						\end{array}
					\label{eqn;instblex}
					\end{equation}
where each $\theta_{ij}$ is as before and take $q\rightarrow0$. The box $\mathcal{B}_{00}$ for~(\ref{eqn;instbl}) corresponds to the regular domain $\mathcal{D}_{00}=\mathbb{R}_{+}^2 \times \mathcal{B}_{00}$ for~(\ref{eqn;instblex}), and the fixed point $(0,0,0,0)$ for~(\ref{eqn;instblex}) corresponds to $(0,0)$ for~(\ref{eqn;instbl}). In $\mathcal{D}_{00}$ with initial conditions $x_{1}(0)=x_{1}^{(0)}$, $y_{1}(0)=y_{1}^{(0)}$, $x_{2}(0)=x_{2}^{(0)}$, and $y_{2}(0)=y_{2}^{(0)}$, (\ref{eqn;instblex}) has solutions
					\begin{equation*}
						\begin{array}{ll}
							x_{1}(t)=x_{1}^{(0)}e^{-\frac{7}{8}t} \\
							y_{1}(t)=8x_{1}^{(0)}e^{-\frac{7}{8}t}+(y_{1}^{(0)}-8x_{1}^{(0)})e^{-t} \\
							x_{2}(t)=x_{2}^{(0)}e^{-\frac{7}{8}t} \\
							y_{2}(t)=\frac{8}{9}x_{2}^{(0)}e^{-\frac{7}{8}t}+(y_{1}^{(0)}-\frac{8}{9}x_{1}^{(0)})e^{-2t}.
						\end{array}
					\end{equation*}
				After going through the process described in Section \ref{subsec:fixedpoints}, we find that $\Gamma_{u_{1}}^{(0)}$ is given parametrically by 
					\begin{equation*}
						\Bigg\{ 
						\begin{array}{l}
							x_{1}(t)=\left(\frac{8}{7}\right)^{7}e^{-\frac{7}{8}t} \\
							y_{1}(t)=8\left(\frac{8}{7}\right)^{7}(e^{-\frac{7}{8}t}-e^{-t})
						\end{array},\ \ 
						0 \leq t \leq 8\ln\left(\frac{8}{7}\right).
					\end{equation*}
				and that $\Gamma_{u_{2}}^{(0)}$ is given parametrically by 
					\begin{equation*}
						\Bigg\{ 
						\begin{array}{l}
							x_{2}(t)=\left(\frac{16}{7}\right)^{7/9}e^{-\frac{7}{8}t} \\
							y_{1}(t)=\frac{8}{9}\left(\frac{16}{7}\right)^{7/9}(e^{-\frac{7}{8}t}-e^{-2t})
						\end{array},\ \ 
						0 \leq t \leq \frac{8}{9}\ln\left(\frac{16}{7}\right).
					\end{equation*}
				By definition, $\mathcal{R}^{(00)}_{c}=\mathcal{R}^{(0)}_{c_{1}}\times\mathcal{R}^{(0)}_{c_{2}}$. The regions $\mathcal{R}^{(0)}_{c_{1}}$ and $\mathcal{R}^{(0)}_{c_{2}}$ are shown in Figure~\ref{fig:convergencedom}. 
				
				A loss of global stability with respect to the regular domain is seen when we compare a point in $\mathcal{B}_{00}$ and a corresponding point in $\mathcal{D}_{00}$. For instance, compare the points $P=(\frac{1}{2},\frac{1}{4})$ and $Q=(3,\frac{1}{2},1,\frac{1}{4})$. Since $P\in \mathcal{B}_{00}$, the trajectory starting at $P$ converges to $(0,0)$. However, the trajectory starting from the corresponding point $Q$ does not converge to $(0,0,0,0)$ by Proposition~\ref{fixpointprop} since $Q \notin \mathcal{R}^{00}_{c}$. Thus, $(0,0,0,0)$ is no longer globally asymptotically stable with respect to $\mathcal{D}_{00}$, although $(0,0,0,0)$ is still locally asymptotically stable by Proposition~\ref{stableprop}. 
					\begin{figure}[!htb]
				\hspace*{-1.3cm}
					\includegraphics[scale=.37]{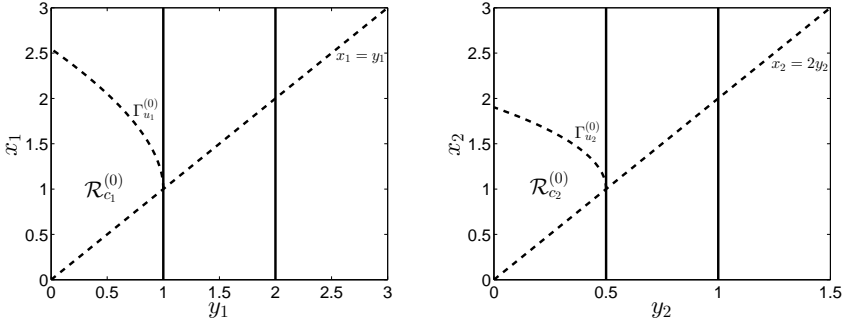}
					\caption{Phase plane of Equations~(\ref{eqn;instblex}) showing $\mathcal{R}^{(0)}_{c_{1}}$ and $\mathcal{R}^{(0)}_{c_{2}}$.}
				\label{fig:convergencedom}
				\end{figure} 
			\label{prop2ex} \qed
			\end{exmp}			
					
		\subsection{Non-Uniqueness}\label{subsec:nonunique}
		
		We motivate the work in this section by the following example.
		\begin{exmp}
			Consider the 1-gene network
				\begin{equation}
					\begin{array}{l}
						\dot{x}=Z- x \\
						\dot{y}=3 x - y
					\end{array}
				\label{eqn;nonuniquesys}
				\end{equation}
				with $\theta=2$. There are fixed points $(0,0)$ and $(1,3)$ in the regular domains. It is a simple application of the singular perturbation theory to show that there is in addition a fixed point near $(x,y)=(\tfrac{2}{3},2)$ for $q>0$ sufficiently small. In the limit $q\to 0$, this lies on the threshold. Taking $q \rightarrow 0$, the solutions in $\mathcal{D}_{1}$ with initial conditions $x(0)=x^{(0)}$ and $y(0)=y^{(0)}$ are given by
				\begin{equation*}
					\begin{array}{l}
						x(t)=1+(x^{(0)}-1)e^{-t} \\
						y(t)=3+(3x^{(0)}-3)te^{-t}+(y^{(0)}-3)e^{-t}\,.
					\end{array}
				\end{equation*}
				For a solution starting at $(0,3\ln(3))$ we notice that at time $t=\ln(3)$ 
				\begin{equation*}
					\begin{array}{l}
						x(\ln(3))=1-e^{-\ln(3)}=\frac{2}{3} \\
						y(\ln(3))=3-3\ln(3)e^{-\ln(3)}+(3\ln(3)-3)e^{-\ln(3)}=2\,,
					\end{array}
				\end{equation*}
				which is the location of the fixed point. Thus, for any solution starting on the stable manifold in $\mathcal{D}_{1}$, it takes at most $\ln(3)$ time units to reach the fixed point. 
						
				Similarly, in $\mathcal{D}_{0}$, solutions with initial conditions $x(0)=x^{(0)}$ and $y(0)=y^{(0)}$ are given by 
				\begin{equation*}
					\begin{array}{l}
						x(t)=x^{(0)}e^{-t} \\
						y(t)= 3x^{(0)}te^{-t}+y^{(0)}e^{-t}\,.
					\end{array}
				\end{equation*} 
				For a solution starting at $(\frac{2}{3}e,0)$ we notice that at time $t=1$
				\begin{equation*}
					\begin{array}{l}
						x(1)=\frac{2}{3}e\cdot e^{-1}=\frac{2}{3} \\
						y(1)= 3\frac{2}{3}e\cdot e^{-1}=2,
					\end{array}
				\end{equation*}
				which is also the location of the fixed point. It is clear that uniqueness is now lost and as a result what happens to these flows after they intersect $(\frac{2}{3},2)$ is ambiguous. Either they follow either of the unstable manifolds, or will remain at the fixed point. \qed
					\begin{figure}[!htb]
						\hspace*{-1.5cm}
						\includegraphics[scale=.4]{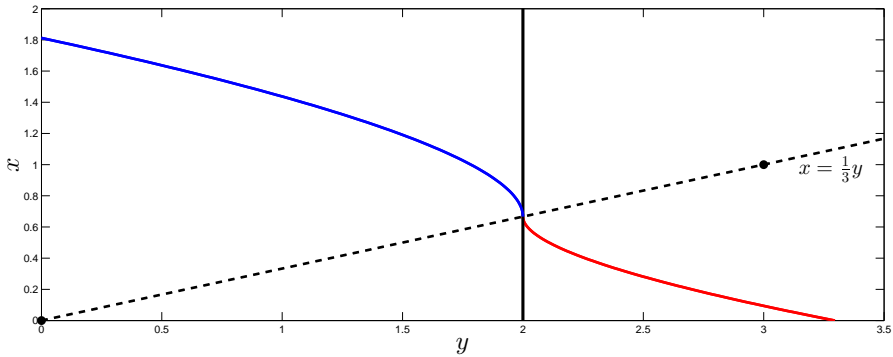}
						\caption{The non-unique solutions of (\ref{eqn;nonuniquesys}), shown in red and blue.}
					\label{fig:nonunique}
					\end{figure}
				\end{exmp}
				
		It is natural to ask when this happens. That is, under what conditions is there a non-unique solution in the limit $q\rightarrow 0$? In this section we give an answer to this question. We proceed with the following assumption:
			\begin{newass}
				At least one $F_{i}$ depends on $Z_{ij}$ for at least one $j$.
			\end{newass}
				
		Assumption 2 says that there is autoregulation. For simplicity, we take $i=1$ in what follows, but the conclusion holds for any $i$. Suppose that we are in a regular domain $\mathcal{D}_{h}$ such that $y_{1}$ is the next variable to hit a threshold (this is not unreasonable, as it could happen, for instance, that each other pair $(x_{j},y_{j})$ for $i\neq 1$ has focal point coordinates in the regular domain $\mathcal{D}_{h}$). This allows us to disregard all variables with $i>1$, as each $Z_{ij}$ will be constant (0 or 1) in the limit $q\to 0$ (for $q>0$ sufficiently small, it will be almost constant, but the result will take more work!). For simplicity again, we assume that $y_{1}$ has only one threshold, but this does not change the conclusion of the result as we are only concerned with the dynamics close to one threshold. 
				
		Since we take each $F_{i}$ to be a multilinear polynomial, close to a threshold $\theta_{1,1}$, we can write (\ref{def;Model2}) as
			\begin{equation}
				\begin{array}{l}
					\dot{x}_{1}=b+aZ_{11} -\beta_{1}x_{1} \\
					\dot{y}_{1}=\kappa_{1}x_{1}-\gamma_1 y_{1} \\
					\dot{x}_{i}=F_{i}(Z)-\beta_{i}x_{i} \\
					\dot{y}_{i}=\kappa_{i}x_{i}-\gamma_i y_{i}
				\end{array} i=2,\dots, n
			\label{eqn;nonunique}
			\end{equation}
where $b\geq 0$ and $a\geq -b$ are constants. For this section, suppose that $a\geq 0$. If $-b\leq a< 0$, then it can easily be shown that the dynamics we wish to explore do not take place. 
				
		In what follows we drop the subscripts for all variables and parameters, as it is understood that they are all 1.
				
			\begin{newthm}
				If $a < \frac{\beta\gamma}{\kappa}\theta< a+b$, then there is a fixed point for the pair $(x,y)$ in a neighbourhood of $(\frac{\gamma}{\kappa}\theta,\theta)$. Moreover, this point converges to $(\frac{\gamma}{\kappa}\theta,\theta)$ as $q\rightarrow 0$. If $q$ is sufficiently small, then this fixed point is a saddle point.
			\end{newthm}
			\begin{proof}
				We need to show that there exists a point $(\bar{x},\bar{y})$ close to $\left(\frac{\gamma}{\kappa}\theta,\theta\right)$ such that 
					\begin{equation}
					\label{eqn; sys666}
						\begin{array}{l}
							0=b+a\frac{\bar{y}^{\frac{1}{q}}}{\bar{y}^{\frac{1}{q}}+\theta^{\frac{1}{q}}}-\beta \bar{x} \\
							0=\kappa \bar{x}-\gamma \bar{y} \\
						\end{array}
					\end{equation} 
				is satisfied. Define 
					\begin{equation*}
						p(x):=\left(\dfrac{\kappa x}{\gamma \theta}\right)^{\frac{1}{q}}-\dfrac{\beta x -b}{(a+b)-\beta x}.
					\end{equation*}
				Since 
					\begin{displaymath}
						p\left( \frac{b}{\beta}\right)>0 \ \textup{and} \ \lim_{x \rightarrow \frac{b+a}{\beta}^{-}}p(x)=-\infty,
					\end{displaymath} 
				there exists an $\bar{x} \in \left(\frac{b}{\beta}, \frac{b+a}{\beta} \right)$ such that $p(\bar{x})=0$, by the intermediate value theorem. It follows that $\bar{y}=\frac{\kappa}{\gamma}\bar{x}\in\left(\frac{\kappa b}{\gamma \beta},\frac{\kappa(a+b)}{\gamma \beta}\right)$. 
				
				By (\ref{eqn; sys666}),
					\begin{equation}
						\bar{y}=\left(\dfrac{\beta \bar{x}-b}{(a+b)-\beta \bar{x}}\right)^{q}\theta
					\label{eqn; limity}
					\end{equation} 
				so it follows that 
					\begin{equation*}
						\lim_{q\rightarrow 0}(\bar{x},\bar{y})=\left(\frac{\gamma}{\kappa}\theta,\theta\right).
					\end{equation*}	
				
				At this point, we are interested in the eigenvalues of the linearization of (\ref{eqn; sys666}) about $(\bar{x},\bar{y})$. We have 
					\begin{equation*}
						J(\bar{x},\bar{y})=
							\left(\begin{array}{cc}
								-\beta & \dfrac{a\theta^{\frac{1}{q}}\bar{y}^{\frac{1}{q}-1}}{q\left(\bar{y}^{\frac{1}{q}}+\theta^{\frac{1}{q}}\right)^{2}} \\ 
								\kappa & -\gamma
							\end{array}\right),
					\end{equation*}
				which has eigenvalues
					\begin{equation*}
						\lambda_{1,2}=-\frac{\beta + \gamma}{2} \pm \sqrt{\frac{(\beta +\gamma)^{2}}{4}+\dfrac{a\kappa\theta^{\frac{1}{q}}\bar{y}^{\frac{1}{q}-1}}{q\left(\bar{y}^{\frac{1}{q}}+\theta^{\frac{1}{q}}\right)^{2}}-\beta \gamma}.
					\end{equation*}
					
				Recall (\ref{eqn; limity}) and note that the limit 
					\begin{align*}
						\lim_{q\rightarrow 0}\frac{a\kappa\theta^{\frac{1}{q}}\bar{y}^{\frac{1}{q}-1}}{q(\bar{y}^{\frac{1}{q}}+\theta^{\frac{1}{q}})^{2}}=&\lim_{q\rightarrow0}\frac{a\kappa\theta^{\frac{1}{q}}\left(\dfrac{\beta \bar{x}-b}{(a+b)-\beta \bar{x}}\right)\theta^{\frac{1}{q}}}{q\left(\left(\dfrac{\beta \bar{x}-b}{(a+b)-\beta \bar{x}}\right)\theta^{\frac{1}{q}}+\theta^{\frac{1}{q}}\right)^{2}\bar{y}} \\ \\=& 
						\lim_{q\rightarrow0}\frac{a\kappa\left(\dfrac{\beta \bar{x}-b}{(a+b)-\beta \bar{x}}\right)}{q\left(\left(\dfrac{\beta \bar{x}-b}{(a+b)-\beta \bar{x}}\right)+1\right)^{2}\bar{y}}=\infty
					\end{align*}
				since $a < \frac{\beta_{1}\gamma_{1}}{\kappa_{1}}\theta_{1}< a+b$, $\bar{x}\rightarrow \frac{\gamma \theta}{\kappa}$, and $\bar{y}\rightarrow \theta$. From this we can conclude that this point is a saddle, because 
					\begin{equation*}
						\frac{a\kappa\theta^{\frac{1}{q}}\bar{y}^{\frac{1}{q}-1}}{q\left(\bar{y}^{\frac{1}{q}}+\theta^{\frac{1}{q}}\right)^{2}}-\beta \gamma>0
					\end{equation*}
				for sufficiently small $q$, as shown above. 
			\end{proof} 
				
		If we define solutions of the step function system to be the limit as $q \rightarrow 0$ of solutions of the smooth system, then the stable manifolds of the fixed point for the step function system are given by the curves $\Gamma^{(1)}_{l_{1}}$ and $\Gamma^{(0)}_{u_{1}}$ described in Section~\ref{subsec:fixedpoints}. 
						
		Thus, we can heuristically explain the phenomenon of flows reaching the fixed point in finite time, since the eigenvalues of $J(\bar{x},\bar{y})$ go to positive and negative infinity as $q \rightarrow 0$. Thus, even though it takes an infinite amount of time for flows to reach an equilibrium point for any $q>0$ with approach proportional to $e^{\lambda t}$, as $q \rightarrow 0$ we have that $\lambda\rightarrow-\infty$ and the solution can reach the equilibrium in finite time. 
					
		If an orbit of the step function system were to start on the stable manifold, then it can reach the fixed point in finite time, as shown in the following example. Once the trajectory intersects the fixed point the dynamics become ambiguous. This phenomena is analogous to the case where solutions pass through the intersection of two thresholds in the framework of Model 1, as shown by Killough and Edwards \cite{ek2005}.
					
		
	\section{Dynamics Through a Series of Regular Domains}\label{sec:globaldynam}
		In the previous section we discussed the flow within a regular domain. We now use this flow to determine a mapping $M:\mathbb{R}^{2n}_{+}\rightarrow \mathbb{R}^{2n}_{+}$ from one threshold to another, that is, $(x^{(m+1)},y^{(m+1)})=M(x^{(m)},y^{(m)})$, in the limit $q\to 0$. This has been done in~\cite{Edwards2014} for the case $\beta_{i} \neq \gamma_{i}$, but here we also allow equality. As before, with $N=\{1,\dots,n\}$, let $I=\{i\in N\ |\ \beta_{i}=\gamma_{i}\}$ and $J=N\setminus I=\{i\in N\ |\ \beta_{i}\neq \gamma_{i}\}$. In the rest of this section, take $i\in I$ and $j\in J$ and take $q\to 0$.
		
		Suppose that a solution passes through a wall $y_{s_{m}}=\theta_{s_{m}}$ at time $T^{(m)}$, at the point $y_{s_{m}}(T^{(m)})= \theta_{s_{m}}$, $x_{i}(T^{(m)})=x_{i}^{(m)}$, $y_{i}(T^{(m)})=y_{i}^{(m)}$, $x_{j}(T^{(m)})=x_{j}^{(m)}$ and $y_{j}(T^{(m)})=y_{j}^{(m)}$ and proceeds into a regular domain $\mathcal{D}_{h}^{(m)}$, where the index $(m)$ is a counter for the intersections of the solution and walls.  We wish to determine on which wall the trajectory exits $\mathcal{D}_{h}^{(m)}$ and the location and time of the next wall intersection. This is done by calculating the next time at which each protein variable would hit a threshold, independent of other variables. The minimum of these times identifies the variable that switches next.
		
		In each direction $\ell\in N$, the walls of $\mathcal{D}_{h}^{(m)}$ either have $y_{\ell}=\theta_{\ell, h_{\ell}}$ or $y_{\ell}=\theta_{\ell, h_{\ell}+1}$. From (\ref{eqn;regysol}), the next intersection point considering direction $\ell$ independently is one of the following (depending on whether or not $\beta_{\ell}=\gamma_{\ell}$), with either $r_{\ell}=h_{\ell}$ or $r_{\ell}=h_{\ell}+1$:
			\begin{equation}
				\begin{array}{l}
					\theta_{\ell,r_{\ell}}=\frac{\kappa_{\ell}\alpha_{\ell}}{\beta_{\ell}\gamma_{\ell}}+\frac{\kappa_{\ell}(x_{\ell}^{(m)}-\alpha_{\ell}/\beta_{\ell})}{\gamma_{\ell}-\beta_{\ell}}e^{-\beta_{\ell}t_{\ell}}+ 
					(y_{\ell}^{(m)}-\frac{\kappa_{\ell}\alpha_{\ell}}{\beta_{\ell}\gamma_{\ell}}-\frac{\kappa_{\ell}(x_{\ell}^{(m)}-\alpha_{\ell}/\beta_{\ell})}{\gamma_{\ell}-\beta_{\ell}})e^{-\gamma_{\ell}t_{\ell}}\,, \\
					\theta_{\ell,r_{\ell}}=\frac{\kappa_{\ell}\alpha_{\ell}}{\gamma_{\ell}^{2}}+(\kappa_{\ell}x_{\ell}^{(m)}-\frac{\kappa_{\ell}\alpha_{\ell}}{\gamma_{\ell}})t_{\ell}e^{-\gamma_{\ell}t_{\ell}}+ 
					(y_{\ell}^{(m)}-\frac{\kappa_{\ell}\alpha_{\ell}}{\gamma_{\ell}^{2}})e^{-\gamma_{\ell}t_{\ell}}\,, \\
				\end{array}
			\label{eqn;thetaeqn}
			\end{equation}
		where $t_{\ell}$ is defined by $y_{\ell}(t_{\ell})=\theta_{\ell,r_{\ell}}$, when it exists. 
					
		As in~\cite{Edwards2014}, we introduce the variable substitution $a_{0}=\frac{\kappa_{\ell}\alpha_{\ell}}{\beta_{\ell}\gamma_{\ell}}-\theta_{\ell,r_{\ell}}$, $b_{0}=\frac{\kappa_{\ell}}{\gamma_{\ell}-\beta_{\ell}}(x_{\ell}^{(m)}-\alpha_{\ell}/\beta_{\ell})$, $c_{0}=y_{\ell}^{(m)}-\frac{\kappa_{\ell}\alpha_{\ell}}{\beta_{\ell}\gamma_{\ell}}-\frac{\kappa_{\ell}}{\gamma_{\ell}-\beta_{\ell}}(x_{\ell}^{(m)}-\alpha_{\ell}/\beta_{\ell})$, $a_{1}=\frac{\kappa_{\ell}\alpha_{\ell}}{\gamma_{\ell}^{2}}-\theta_{\ell,r_{\ell}}$, $b_{1}=\kappa_{\ell}x_{\ell}^{(m)}-\frac{\kappa_{\ell}\alpha_{\ell}}{\gamma_{\ell}}$, and $c_{1}=y_{\ell}^{(m)}-\frac{\kappa_{\ell}\alpha_{\ell}}{\gamma_{\ell}^{2}}$ then we can write the equations in (\ref{eqn;thetaeqn}) as 
			\begin{equation}
				\begin{array}{l}
					0=a_{0}+b_{0}e^{-\beta_{\ell}t_{\ell}}+c_{0}e^{\gamma_{\ell}t_{\ell}}\,, \ \ \text{or} \\
					0=a_{1}+(b_{1}t_{\ell}+c_{1})e^{-\gamma_{\ell}t_{\ell}}.
				\end{array}
			\label{eqn;thetasol}
			\end{equation}
		Note that the second equation in~(\ref{eqn;thetasol}) is transcendental and the first equation is usually transcendental, for instance if $\gamma_{\ell}/\beta_{\ell} \notin \mathbb{Q}$. There is a pair of such equations~(\ref{eqn;thetasol}), for each choice of $r_{\ell}$ (either $h_{\ell}$ or $h_{\ell}+1$). Thus, in principle, for each $\ell$, there could be multiple solutions $t_{\ell}$ to~(\ref{eqn;thetasol}), thought they need not be positive.  
		
		If, for a particular $\ell$, there exists one or more positive solutions of (\ref{eqn;thetasol}), then let $t_{\ell}^*$ be the minimum positive solution. This is the next hitting time in the $\ell$ direction. If there are no positive solutions, let $t_{\ell}^*=\infty$.
		Now, if $t_{\ell}^*<\infty$ for some $\ell$ then the index $s_{m+1}$ specified by 
$$t^{*}_{s_{m+1}}=\min_{\ell}\{t^{*}_{\ell}\}$$ 
indicates the next wall intersection of the trajectory, $y_{s_{m+1}}=\theta_{s_{m+1}}$ (either $\theta_{s_{m+1},h_{s_{m+1}}}$ or $\theta_{s_{m+1},h_{s_{m+1}}+1}$), at step $(m+1)$ and time $T^{(m+1)}=t^{*}_{s_{m+1}}$. We note that equation~(\ref{eqn;thetasol}) may have to be solved numerically. 
				
		If no such solution exists ($t^*_{s_{m+1}}=\infty$) then the solution stays in the regular domain $\mathcal{D}_{h}^{(m)}$ as $t\rightarrow\infty$ and 
			\begin{displaymath}
					\begin{array}{l}
						x_{i}(t)\rightarrow \alpha_{i}/\beta_{i}\,, \\
						x_{j}(t)\rightarrow \alpha_{j}/\beta_{j}\,, \\
						y_{i}(t)\rightarrow \frac{\kappa_{i}f_{i}}{g_{i}\gamma_{i}}\,,\\
						y_{j}(t)\rightarrow \frac{\kappa_{j}f_{j}}{\gamma_{j}^{2}}\,.
					\end{array}
			\end{displaymath}
				
		When the solution $T^{(m+1)}(x_{s_{m}}^{(m)},y_{s_{m}}^{(m)})$ of (\ref{eqn;thetasol}) is inserted in (\ref{eqn;regxsol}) and (\ref{eqn;regysol}) for $t$, the remaining coordinates $x_{i}$, $y_{i}$, $x_{j}$, and $y_{j}$, $i\in I$, $j \in J$ of the intersection point $(x^{(m+1)},y^{(m+1)})$ with the wall $y_{s_{m}}=\theta_{s_{m+1}}$ can be found as a function of $x^{(m)}$ and $y^{(m)}$. 
		
		Thus, we have described a mapping $M:\mathbb{R}^{2n}_{+}\rightarrow \mathbb{R}^{2n}_{+}$ from one threshold to another. One could write a computer program to iterate this process and analytically integrate trajectories indefinitely. In previous work on Model 1, one could attain a closed form of the mapping from one wall to another, and thus, explicitly calculate returning regions and find conditions for the existence and stability of periodic orbits. Due to the transcendental nature of the solutions in Model 2, the equivalent calculations would involve non-trivial numerical computations. 
		
		In Model 1 the time to next switching, $T^{(m+1)}-T^{(m)}$, varies continuously with respect to the current switching point $y^{(m)}$. However, in Model 2 this is not true when an initial coordinate $(x_{i}^{(m)},y_{i}^{(m)})$ is moved through the curve $\Gamma_{u_{i}}^{(h_{i})}$ or  $\Gamma_{\ell_{i}}^{(h_{i})}$, at which point $T^{(m+1)}(x_{s_{m}}^{(m)},y_{s_{m}}^{(m)})$ has a discontinuity. On one side of the curve, the trajectory passes through the threshold. On the other side, it turns and comes back before crossing, and thus crosses later at another threshold.
		
		Suppose we enter a regular domain $\mathcal{D}_{h}$ at time $t^{(m)}$. If $\frac{\alpha_{i}}{\beta_{i}}<\frac{\gamma_{i}}{\kappa_{i}}\theta_{i,h_{i}}$, $\frac{\alpha_{i}}{\beta_{i}}>\frac{\gamma_{i}}{\kappa_{i}}\theta_{i,h_{i}+1}$, or $\frac{\gamma_{i}}{\kappa_{i}}\theta_{i,h_{i}}<\frac{\alpha_{i}}{\beta_{i}}<\frac{\gamma_{i}}{\kappa_{i}}\theta_{i,h_{i}+1}$ for some $i$, then $\Gamma^{(h_{i})}_{u_{i}}$, $\Gamma^{(h_{i})}_{\ell_{i}}$, or $\Gamma^{(h_{i})}_{u_{i}}$ and $\Gamma^{(h_{i})}_{\ell_{i}}$ exist respectively and can be found by using the method given in Section~\ref{subsec:fixedpoints}. Given the initial condition $y_{i}(t^{(m)})=y_{i}^{(m)}$, let $(x^{*}_{i},y_{i}^{(m)})$ be the point where $y_{i}=y_{i}^{(m)}$ intersects $\Gamma^{(h_{i})}_{u_{i}}$ and let $(x'_{i},y_{i}^{(m)})$ be the point where $y_{i}=y_{i}^{(m)}$ intersects $\Gamma^{(h_{i})}_{\ell_{i}}$ . We can implicitly describe the time it takes for $y_{i}$ to reach it's next threshold as a function of initial condition $x_{i}^{(m)}$. By Equation~(\ref{eqn;regysol}), the time $t$ it takes for $y_{i}$ to cross a threshold is then given implicitly by 
			\begin{equation*}
				 \Bigg \{ 
					\begin{array}{lcl}
						\theta_{i,h_{i}}=y_{i}(t;x_{i}^{(m)},y_{i}^{(m)}) & \textup{if} & x_{i}^{(m)}<x_{i}^{*} \\
						\theta_{i,h_{i}+1}=y_{i}(t;x_{i}^{(m)},y_{i}^{(m)}) & \textup{if} & x_{i}^{(m)}>x_{i}^{*}
					\end{array}
			\end{equation*}
		Although these equations are transcendental, they can easily be solved numerically. It is now clear that there is a discontinuity at the point $x_{i}^{(m)}=x_{i}^{*}$ or $x_{i}^{(m)}=x_{i}'$, accordingly.
		
		By the comments after the proof of Proposition~\ref{fixpointprop}, we note that a time discontinuity still occurs when $q>0$, although the exact calculation of the time map is harder to compute explicitly. In the case discussed in Section~\ref{subsec:nonunique}, however, it is clear that the time discontinuity occurs at the fixed point near the threshold, and the stable manifold represents an asymptote in the plot of the time to the next switching.

			\begin{exmp}
				Consider the 1-gene network 
					\begin{equation*}
						\begin{array}{l}
							\dot{x}=1-Z_{1}+Z_{2}-x \\
							\dot{y}=3x-y
						\end{array}
					\end{equation*}
				with $\theta_{1}=1$ and $\theta_{2}=2$. We are interested in finding the time $t$ it takes for $y(t)$ to cross the next threshold from $\theta_{1}$. Since $\frac{\alpha}{\beta}=0<\frac{1}{3}=\frac{\gamma}{\kappa}\theta_{1}$ in $\mathcal{D}_{1}$, the time to next switching $t(x^{(0)})$ will be determined by $y_{i}(t;x_{i}^{(0)},y_{i}^{(0)})$. 
				
				Using the method described in Section~\ref{subsec:fixedpoints}, we find that $x^{*}$ is the solution of 
					\begin{equation*}
						1=\ln\left(\frac{3x^{*}}{2}\right)+\frac{1}{3x^{*}}
					\end{equation*}
				which is approximately 1.43702. Thus, having that  
					\begin{equation*}
						y_{i}(t;x_{i}^{(0)},y_{i}^{(0)})=(3x^{(0)}t+1)e^{-t}
					\end{equation*}
				it follows that $t(x^{(0)})$ is given implicitly by
					\begin{equation}
						\Bigg \{
						\begin{array}{lcl}
							1=(3x^{(0)}t+1)e^{-t} & \textup{if} & \frac{1}{3}<x^{(0)}<x^{*}\approxeq 1.43702 \\
							2=(3x^{(0)}t+1)e^{-t} & \textup{if} & x^{(0)}>x^{*}\approxeq 1.43702
						\end{array}
					\label{timefunc}
					\end{equation} 
				(see Figure~\ref{fig:timeplot}). We must have that $x^{(0)}\geq \frac{1}{3}$ because otherwise the flow is exiting the domain rather than entering.\qed
					\begin{figure}[!htb]
						\hspace{-1.3cm}
						\includegraphics[scale=.45]{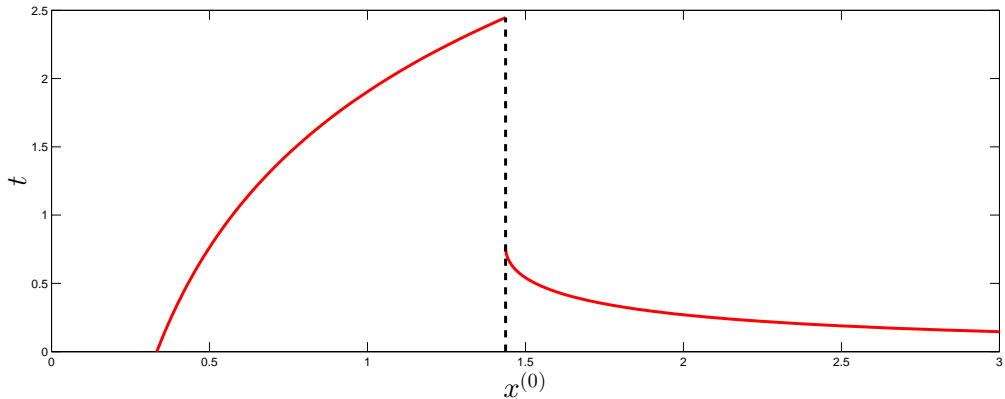}
						\caption{Plot of (\ref{timefunc}): time for $y$ to intersect a threshold as a function of $x^{(0)}$, the initial position on $y=\theta_{1}=1$.}
					\label{fig:timeplot}
					\end{figure}
								
			\end{exmp}		
	\section{Symbolic Representation}\label{sec:graphrep}
		Here we extend the familiar concept of a state transition diagram to the expanded network. The results in this section will be useful in the following section. 
		
		In order to better discuss the qualitative dynamics of Model 2, we encode the states of the system onto a directed graph. Each vertex represents a specific region of phase space (defined below), and the direction along each edge describes the direction of the flow from one domain to another. For this, the concept of a regular domain is too general, so we define another division of phase space into what we call Pseudo-State Domains. With $h,j \in \mathcal{H}$, we define a \textit{Pseudo-State Domain} to be
			\begin{equation*}
				\mathcal{P}_{h,j}=\mathcal{P}_{h_{1},\dots,h_{n},j_{1}\dots,j_{n}}=\prod_{i=1}^{n}(\theta_{i,h_{i}},\theta_{i,h_{i}+1})\times\left(\dfrac{\gamma_i}{\kappa_i}\theta_{i,j_i},\dfrac{\gamma_i}{\kappa_i}\theta_{i,j_i+1}\right).
			\end{equation*}	
		Under this definition it is clear that for a given regular domain $\mathcal{D}_{h}$, $\overline{\mathcal{D}}_{h}=\bigcup_{j=1}^{n}\overline{\mathcal{P}}_{h,j}$, where $\overline{U}$ denotes the closure of a set $U$.
		
		We can describe the activity of the network symbolically by relating it to a directed graph, which we call a pseudo-state transition diagram, or PTD, where we assign a pseudo-state domain $\mathcal{P}_{h,j}$ to each vertex. Then, the direction of the flow through pseudo-state domains is uniquely represented by a direction on the edge between the corresponding vertices. We show that the flow directions are well-defined in the Lemma below.
		
		Recalling that $p_{i}$ is the number of threshold values for $y_{i}$, the PTD has $M=\prod_{i=1}^{n}(p_{i}+1)^{2}$ vertices. 
		
		This digraph represents a state transition diagram encoding the logical structure of a network. Therefore, it defines an equivalence class of networks with the same underlying structure. 
		
			\begin{newlem}
				If $\frac{\alpha_{i}}{\beta_{i}}\neq \frac{\gamma_{i}}{\kappa_{i}}\theta_{i,h_{i}}$ for all $i$, then every edge on the PTD has a unique direction.
			\label{lem;nobranch}
			\end{newlem}
			\begin{proof}
				The flow is transverse to edges in $y_i$ directions because of Assumption~\ref{ass; ass1}. The flow is transverse to edges in $x_i$ by the assumption of this Lemma. To show uniqueness we note that the flow in $x_{i}$ is monotonic towards $\frac{\alpha_{i}}{\beta_{i}}$. In order for $x_{i}$ to cross $\frac{\gamma_{i}}{\kappa_{i}}\theta_{i,h_{i}}$ in both directions we would require that $y_{j}$ cross a threshold for some $j$. But, once this happens we are in a different pseudo-state domain. 
				
				By the flow of (\ref{def;Model2}) or (\ref{eqn;reg}), $y_{i}$ can only cross a threshold $y_{i}=\theta_{i,h_{i}}$ in the increasing direction for $x_{i}>\frac{\gamma_{i}}{\kappa_{i}}\theta_{i,h_{i}}$ and in the decreasing direction for $x_{i}<\frac{\gamma_{i}}{\kappa_{i}}\theta_{i,h_{i}}$. By definition of a pseudo-state domain, it follows that the direction across a threshold is unique.
			\end{proof}
		
	\section{Negative Feedback Loops}\label{sec:negfeed}
			Feedback loops are a basic and important structure in gene regulation, being the essential motif behind oscillatory process such as circadian rhythms and the cell cycle, for example. Proofs of results concerning existence and stability of periodic orbits have been obtained in many negative feedback systems, and in particular in piecewise-linear models for gene regulation~\cite{FG2009,gp78b,Snoussi1989}. These results are not sufficient to deal with the current model framework, with only half the variables involving steep sigmoidal or step function switches. Since this is an important analytic result, albeit only for a particular simple structure, we obtain in this section a comparable result on existence of periodic orbits for Model 2.
			
			We work with the following general form of a feedback loop into which framework Model 2 fits: 
				\begin{equation}
					\begin{array}{ll}
						\begin{array}{l}
							\dot{x}_{1}=-b_{1}x_{1}+ a_{1}f_{1}(y_{n}) \\
							\dot{y}_{1}=-d_{1}y_{1}+c_{1}x_{1} \\
							\dot{x}_{i}=-b_{i}x_{i}+a_{i}f_{i}(y_{i-1})\\
							\dot{y}_{i}=-d_{i}y_{i}+c_{i}x_{i}\,,
						\end{array} &\quad  i=2,\dots, n,
					\end{array}
				\label{eqn;cyclicfeedbackdefn}
				\end{equation}
where constants $a_i, b_i, c_i, d_i$ are taken to be strictly positive for all $i=1,\ldots,n$. We take $f_i\ge 0$ for all $i$, and 
without loss of generality, 
$f_i(y_{i-1})\in [0,1]$ for $y_{i-1}\ge 0$, and similarly for $f_1(y_n)$. In what follows, we will indentify index $0$ with index $n$, so that $y_{i-1}$ is $y_n$ when $i=1$. In our Model 2, we have $f_i(y_{i-1})=\mathcal{S}(y_{i-1},\theta_{i-1},q)$, the sigmoid with $\theta_{i-1}>0$, and the prototypical examples are the Hill functions $H(y_{i-1},\theta_{i-1},q)$ or $1-H$, in which case $S(\theta_{i-1},\theta_{i-1},q)=\frac 12$. We may take this as the definition of the threshold for general $\mathcal{S}$. 

For the following, however, we only require that for all $i=1,\ldots n$, $f_i\in C^1(\mathbb{R})$ is monotone (either increasing or decreasing), with $0\leq f_i \leq 1$, and $f_i(y_{i-1})$ set to $f_i(0)$ for $y_{i-1}<0$ if $f_i(0)=0$ or $1$ (as is the case for the Hill functions, for example). In this case we call (\ref{eqn;cyclicfeedbackdefn}) a {\em monotone feedback loop}.
			
			For the following section we apply the change of variables 
				\begin{equation*}
					\begin{array}{lcl}
						x_{1} & \longmapsto & w_{1} \\
						y_{1} & \longmapsto & w_{2} \\
							 & \ddots & \\
						x_{n} & \longmapsto & w_{2n-1} \\
						y_{n} & \longmapsto & w_{2n},
					\end{array} 
				\end{equation*}
			so that (\ref{eqn;cyclicfeedbackdefn}) becomes 
				\begin{equation}
					\begin{array}{ll}
						\begin{array}{l}
							\dot{w}_{1}=-b_{1}w_{1}+ a_{1}f_{1}(w_{2n}) \\
							\dot{w}_{2}=-d_{1}w_{2}+c_{1}w_{1} \\
							\dot{w}_{i}=-b_{i}w_{i}+a_{i}f_{i}(w_{i-1}) \\
							\dot{w}_{i+1}=-d_{i}w_{i+1}+c_{i}w_{i},
						\end{array} & \quad i = 3,5, \dots, 2n-1,
					\end{array}
				\label{eqn;cfs,chng}
				\end{equation}
				where the indices of the $f_i, a_i, b_i, c_i$ and $d_i$ have been renumbered in the natural way.
			We say that (\ref{eqn;cfs,chng}) is a {\em monotone negative feedback loop} provided that 
				\begin{equation}
					\dfrac{\partial\dot{w}_{i}}{\partial w_{i-1}}<0
				\label{negativefeedbackcondition}
				\end{equation}
			occurs an odd number of times (identifying the index $0$ with $n$). We say that (\ref{eqn;cfs,chng}) is a {\em monotone positive feedback loop} if (\ref{negativefeedbackcondition}) occurs an even number of times. 
					
			\begin{exmp}\label{ex:STD}
				Taking $q\rightarrow 0$, Figure~\ref{fig:STD} is a pseudo-state transition diagram for a 2-gene negative feedback loop 
					\begin{align*}
						\dot{w}_1&=-w_1+1-Z_2(w_4) \\
						\dot{w}_2&=-w_2+w_1 \\
						\dot{w}_3&=-w_3+Z_1(w_2) \\
						\dot{w}_4&=-w_4+w_3 \\
					\end{align*}
				with $\theta_1=\theta_2=2$. Note that all trajectories eventually converge to the domain $0000$, in correspondence with the fixed point $(w_1,w_2,w_3,w_4)=(1,1,0,0)$. 
 					\begin{figure}[!htb]
						\centering
						\includegraphics[scale=.45]{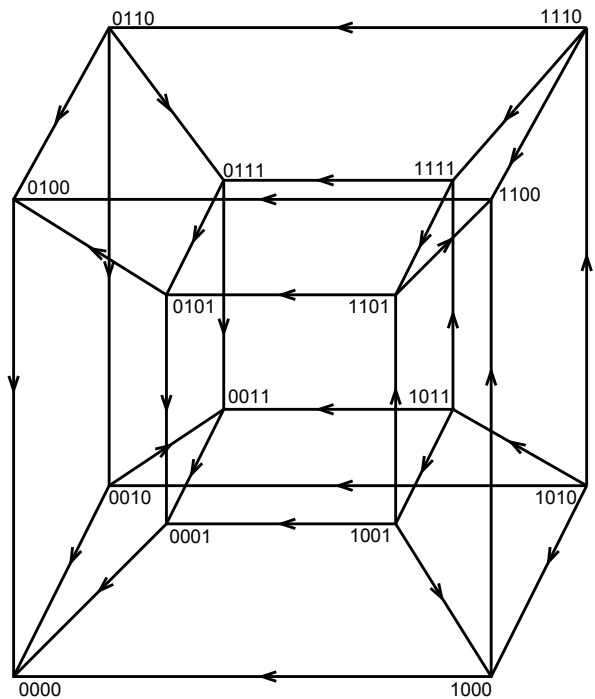}
						\caption{Pseudo-state transition diagram for the negative feedback loop of Example~\ref{ex:STD}}
						\label{fig:STD}
					\end{figure}\qed
			\end{exmp}		

We are mainly interested in monotone negative feedback loops that allow the possibility of oscillation. Each variable in Equation~(\ref{eqn;cyclicfeedbackdefn}) or (\ref{eqn;cfs,chng}) has a range of values in which it must eventually lie. In fact, there is an invariant region for the whole system. Any variable initially outside this range must fall into it. Furthermore, in the case that $f_i(w_{i-1})=\mathcal{S}(w_{i-1},\theta_{i-1},q)$ ($i$ odd), if $\theta_{i-1}$ does not lie in this interval, then this variable cannot switch and oscillation is precluded. The intervals can easily be determined from (\ref{eqn;cfs,chng}) as $w_i\in [0,w_{i,{\rm max}}]$, where
				\begin{equation}
					\begin{array}{l}
							w_{i,{\rm max}}= \frac{a_i}{b_i},\mbox{ if $i$ is odd}, \\
							w_{i,{\rm max}} = \frac{c_{i-1}a_{i-1}}{d_{i-1}b_{i-1}},\mbox{ if $i$ is even}.
					\end{array}
				\label{eqn;invariant}
				\end{equation}
Thus, the class of monotone negative feedback loops in which there is a possibility of oscillation from a structural point of view, is that for which
				\begin{equation}\label{eq:theta}
					0 < \theta_i < w_{i,{\rm max}}, \mbox{ for every even $i$}.
				\end{equation}
We will restrict our attention to this class when we consider cycles in section~\ref{subsec:cycles}, below.

		\subsection{Canceling Method}\label{subsec:cancel}
			In this section we provide a method by which we can change the variables in a monotone negative feedback loop of the form (\ref{eqn;cyclicfeedbackdefn}) or (\ref{eqn;cfs,chng}) such that $(\ref{negativefeedbackcondition})$ occurs only once, and such that if (\ref{eq:theta}) holds initially, then it is conserved. This will be useful in proving a result on existence of periodic solutions, using a previously established theorem. This change of variables idea has certainly been proposed before (for example, by Mallet-Paret and Smith~\cite{MPS1990}), but here we need to ensure that our changes of variables keep the system within the well-defined class given by (\ref{eqn;cyclicfeedbackdefn}) or (\ref{eqn;cfs,chng}).
				
			We represent the system with a cycle graph $\mathcal{C}$ with vertices $1,2,\dots,2n$. Each vertex corresponds to a variable $w_{i}$, and each edge represents the influence of $w_{i-1}$ on $w_{i}$. We define an edge $(w_{i-1},w_i)$ to be \textit{negative} and label it ``--'' if (\ref{negativefeedbackcondition}) holds, and we define an edge $(w_{i-1},w_i)$ to be \textit{positive} and label it ``+'' if $\frac{\partial \dot{w}_{i}}{\partial w_{i-1}}>0$ holds. Thus, in a negative feedback loop we have an odd number of negative edges (see Figure~\ref{fig:cycle}).
				
			Let $\mathcal{C}$ be a cycle corresponding to a negative feedback loop. If necessary, change the indices so that the edge $(w_{2n}, w_{1})$ is negative, and label the other $2n-1$ edges positive or negative, accordingly. Note that initially, edges $(w_i,w_{i+1})$ can only be negative if $i$ is even.
				 
				\begin{figure}
				\centering
				\begin{tikzpicture}
				[->,>=stealth',shorten >=1pt,auto,node distance=2cm,thick,main node/.style={circle,fill=red!20,draw,font=\sffamily\large\bfseries}]
				
				  \node[main node] (1) {$w_{1}$};
				  \node[main node] (2) [below left of=1] {$w_{2}$};
				  \node[main node] (3) [below of=2] {$w_{3}$};
				  \node[main node] (4) [below right of=3] {$w_{4}$};
				  \node[main node] (5) [right of=4] {$w_{5}$};
				  \node[main node] (6) [above right of=5] {$w_{6}$};
				  \node[main node] (7) [above of=6] {$\dots$};
				  \node[main node] (8) [right of=1] {$w_{2n}$};
				  	
				  \path[every node/.style={font=\sffamily\small}]
				    (1) edge[bend right] node[left]  {+} (2)
				    (2) edge [bend right] node[left] {--} (3)
				    (3) edge [bend right] node[right] {+} (4)
				    (4) edge [bend right] node[right, yshift=.25cm] {--} (5)
				    (5) edge [bend right] node[right] {+} (6)
				    (6) edge [bend right] node[right] {+} (7)
				    (7) edge [bend right] node[right] {+} (8)
				    (8) edge [bend right] node[right, yshift=.25cm] {--} (1); 
			\end{tikzpicture}
					\caption{A negative feedback loop.}
					\label{fig:cycle}
			\end{figure}
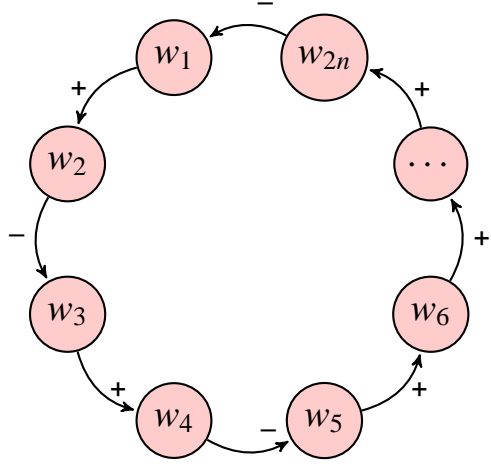
					
			We define the following rule for change of variables: if an edge $(w_{i},w_{i+1})$ is negative, then we make the change of variables $\tilde{w}_i=w_{i,{\rm max}}-w_{i}$. Let $i_{1}$ be the largest index, $i_{1}\leq 2n$, such that $(w_{i_{1}},w_{i_{1}+1})$ has a negative label. By observation $i_1$ is even, so applying our change of variables, we get
				\begin{equation*}
						\begin{array}{ll}
							\dot{\tilde{w}}_{i_1}&=- d_{i_1-1}\tilde{w}_{i_1}-c_{i_1-1}w_{i_1-1}+d_{i_1-1}w_{i_1,{\rm max}},\\
							\dot{w}_{i_1+1}&=-b_{i_1+1}w_{i_1+1}+a_{i_1+1}\tilde{f}_{i_1+1}(\tilde{w}_{i_1}),
						\end{array}
				\end{equation*}		
				where $\tilde{f}_{i_1+1}(\tilde{w}_{i_1})=f_{i_1+1}(w_{i_1,{\rm max}}-\tilde{w}_{i_1})=f_{i_1+1}(w_{i_1})$ so that $\frac{d\tilde{f}_{i_1+1}}{d\tilde{w}_{i_1}}=-\frac{df_{i_1+1}}{dw_{i_1}}$. In the case of our sigmoids with thresholds, $\tilde{\mathcal{S}}_{i_1}(\tilde{w}_{i_1},\tilde{\theta}_{i_1},q)=\mathcal{S}_{i_1}(w_{i_1},\theta_{i_1},q)$, and $\tilde{\theta}_{i_1}=w_{i_1,{\rm max}}-\theta_{i_1}$.
			The sign of $(\tilde{w}_{i_{1}},w_{i_{1}}+1)$ is now positive, but we have made the sign of $(w_{i_{1}-1},\tilde{w}_{i_{1}})$ negative in the process and added a constant term to the $\tilde{w}_{i_1}$ equation. To amend this, we make the change of variables $\tilde{w}_{i_{1}-1}=w_{i_{1}-1,{\rm max}}-w_{i_{1}-1}$ and arrive at the equations
				\begin{equation*}
						\begin{array}{ll}
							\dot{\tilde{w}}_{i_1-1}&=-b_{i_1-1}\tilde{w}_{i_1-1}+b_{i_1-1}w_{i_1-1,{\rm max}}-a_{i_1-1}f_{i_1-1}(w_{i_1-2}),\\
							\dot{\tilde{w}}_{i_1}&=- d_{i_1-1}\tilde{w}_{i_1}+c_{i_1-1}\tilde{w}_{i_1-1}+d_{i_1-1}w_{i_1,{\rm max}}-c_{i_1-1}w_{i_1-1,{\rm max}},\\
							\dot{w}_{i_1+1}&=-b_{i_1+1}w_{i_1+1}+a_{i_1+1}\tilde{f}_{i_1+1}(\tilde{w}_{i_1}).
						\end{array}
				\end{equation*}
				Noting that $b_{i_1-1}w_{i_1-1,{\rm max}}=a_{i_1-1}$ and that $d_{i_1-1}w_{i_1,{\rm max}}-c_{i_1-1}w_{i_1-1,{\rm max}}=0$, this can be written as
				\begin{equation*}
						\begin{array}{ll}
							\dot{\tilde{w}}_{i_1-1}&=-b_{i_1-1}\tilde{w}_{i_1-1}+a_{i_1-1}\tilde{f}_{i_1-1}(w_{i_1-2}),\\
							\dot{\tilde{w}}_{i_1}&=- d_{i_1-1}\tilde{w}_{i_1}+c_{i_1-1}\tilde{w}_{i_1-1},\\
							\dot{w}_{i_1+1}&=-b_{i_1+1}w_{i_1+1}+a_{i_1+1}\tilde{f}_{i_1+1}(\tilde{w}_{i_1}),
						\end{array}
				\end{equation*}
				where $\tilde{f}_{i_1-1}(w_{i_1-2})=1-f_{i_1-1}(w_{i_1-2})$.
			Thus, $(\tilde{w}_{i_1-1},\tilde{w}_{i_1})$ is now positive. In the case of the thresholded sigmoids, 
			$$\tilde{\mathcal{S}}_{i_1-1}(\tilde{w}_{i_1-2},\theta_{i_1-2},q)=1-\mathcal{S}_{i_1-1}(w_{i_1-2},\theta_{i_1-2},q)\,,$$ 
			with the same threshold, since $1-\mathcal{S}_{i_1-1}(\theta_{i_1-2},\theta_{i_1-2},q)=1-\frac 12=\frac 12$.
			
			Now repeat the process. Recursively, if $i_{j}$ is the largest index such that $(w_{i_{j}-1},w_{i_{j}})$ is a negative edge, switch $(w_{i_{j}-1},w_{i_{j}})$ and $(w_{i_{j}-2},w_{i_{j}-1})$. Since $i_{1}>\dots>i_{j}$, we eventually reach $i_j=1$ and $(w_{i_j-1},w_{i_{r}})=(w_{2n}, w_{1})$, which is the unique negative edge, and we are done.
			
			In order to show that the new system is still of the form~\eqref{eqn;cfs,chng}, we need only verify that $\tilde{f}_{i_1+1}(\tilde{w}_{i_1})$ and $\tilde{f}_{i_1-1}(w_{i_1-2})$ are monotone, $C^1$ on $\mathbb{R}$, and in $[0,1]$. These are all obvious from their definitions in terms of $f_{i_1+1}$ and $f_{i_1-1}$, which are assumed to satisfy these conditions, and the same applies at each step around the cycle. Thus, we still have a system in the form of \eqref{eqn;cfs,chng}. In the case that we have sigmoid functions with thresholds, $\tilde{\mathcal{S}}_{i_1+1}(\tilde{w}_{i_1},\tilde{\theta}_{i_1},q)$ is positive at $\tilde{w}_{i_1}=0$ and reaches $1$ at $\tilde{w}_{i_1}=w_{i_1,{\rm max}}$ if $\mathcal{S}_{i_1+1}(0,\theta_{i_1},q)=1$, but since we then define $\mathcal{S}_{i_1+1}(w_{i_1},\theta_{i_1},q)=1$ for $w_{i_1}<0$, we also have $\tilde{\mathcal{S}}_{i_1+1}(\tilde{w}_{i_1},\tilde{\theta}_{i_1},q)=1$ for $\tilde{w}_{i_1}>w_{i_1,{\rm max}}$. Also, $\tilde{\theta}_{i_1}=w_{i_1,{\rm max}}-\theta_{i_1}\in [0,w_{i_1,{\rm max}}]$ whenever $\theta_{i_1}\in  [0,w_{i_1,{\rm max}}]$. Furthermore, $\tilde{\mathcal{S}}_{i_1-1}(w_{i_1-2},\theta_{i_1-2},q)$ has the same threshold as $S_{i_1-1}(w_{i_1-2},\theta_{i_1-2},q)$. Thus, the new system satisfies \eqref{eq:theta} if system~\eqref{eqn;cfs,chng} does.

			Since we may apply a change of variables to any negative feedback loop to obtain another negative feedback loop with a single negative connection, we assume henceforth that all negative feedbacks loops have a single negative connection.   	
					
		\subsection{Cycles}\label{subsec:cycles}
			In Model 1, trajectories of a negative feedback loop follow a cyclic sequence of boxes in phase space if the thresholds are in appropriate ranges. Here we discuss how a similar phase space structure occurs in Model 2, where solutions cycle through boxes corresponding to pseudo-state domains rather than regular domains.
			
			Let $A=\{0,1\}$. In what follows we describe the pseudo-state of a pair $(x_{i},y_{i})$ by assigning it a word $w\in A^{2}$ according to Table~\ref{table1}. \\ 
				\begin{table}
					\centering
					\begin{tabular}{ |c|c| }
						\hline
						$(x_{i},y_{i})$ & $a\in A^{2}$ \\ \hline
						$x_{i}< \frac{d_i}{c_i}\theta_i$, $y_{i}<\theta_{i}$ & $00$ \\ \hline
						$x_{i}> \frac{d_i}{c_i}\theta_i$, $y_{i}<\theta_{i}$ & $10$ \\ \hline
						$x_{i}> \frac{d_i}{c_i}\theta_i$, $y_{i}>\theta_{i}$ & $11$ \\ \hline
						$x_{i}< \frac{d_i}{c_i}\theta_i$, $y_{i}>\theta_{i}$ & $01$ \\ \hline
					\end{tabular}
					\caption{Symbolic description of pseudo-state.}
					\label{table1}
				\end{table}
			We can extend this to describe the pseudo-state of the whole system by assigning a word $a\in A^{2n}$ to the system so that each pair of symbols $a_{i}a_{i+1}$ corresponds to the state of $(x_{i},y_{i})$ according to Table~\ref{table1}.
			
			To show the existence of cyclic sequences of pseudo-state domains in phase space, we first need to establish a property of trajectories of negative feedback loops.
			
				\begin{newlem}
				\label{lem:nocorners}
					In a monotone negative feedback loop, in the limit $q \to 0$, no pair $(x_i,y_i)$ can pass through the point $(\frac{d_i}{c_i}\theta_i,\theta_i)$ from the domains 00 or 11.
				\end{newlem}
				\begin{proof}
					
					Assume first that $0<\frac{d_i}{c_i}\theta_i<\frac{a_i}{b_i}$, $i\neq 1$, and $(x_i, y_i)$ is initially in the domain 00. We assume as well that $y_{i-1}>\theta_{i-1}$ as otherwise $(x_i, y_i)$ will not leave 00. By existence and uniqueness, there is a unique trajectory that intersects the point $( \frac{d_i}{c_i}\theta_i,\theta_i)$; denote this trajectory by $\Delta$. We claim that this trajectory can never enter 00. Indeed, suppose it does. Pick a point $(a,b)\in 00$ that is on $\Delta$. For any $c$ with $b<c<\theta_i$, there exists a trajectory that intersects $(a,c)$. By uniqueness this trajectory cannot cross $\Delta$, and so by the direction of the flow in the $x_i$ direction, it must leave 00 via $y_i=\theta$. However, this contradicts the flow in the $y_i$ direction, so no such $\Delta$ can exist. The cases in which $i=1$ and for which we start in $11$ are analogous.
					
					On the other hand, if $\frac{a_i}{b_i}< \frac{d_i}{c_i}\theta_i$, then both focal points for $(x_i,y_i)$ lie in 00. If a trajectory starts in this domain, then it will not leave and the claim follows. Otherwise, the claim follows by the previous argument. The case in which we start in $11$ is analogous. 
				\end{proof}
				\begin{newthm}
				\label{prop: cycleprop}
					In the limit $q\rightarrow 0$, there exists a qualitative cycle for a negative feedback system if and only if $0<\theta_i<\frac{c_i}{d_i}\frac{a_i}{b_i}$ for each $i$.
				\end{newthm}
				\begin{proof}
				The assumption that $0<\theta_i<\frac{c_i}{d_i}\frac{a_i}{b_i}$ means precisely that for each $i$ the focal point switches from domain 00 to domain 11 and vice versa. 		
						
					Suppose that a trajectory starts in the pseudo-state domain $00\cdots 00$. Then, every pair $(x_i,y_i)$, $i\neq 1$, is ``switched off" and tends towards the focal point in 00. On the other hand, $(x_1,y_1)$ is ``switched on" and tends towards the focal point in 11, by assumption. By Lemma~\ref{lem:nocorners}, the path it must take to get there is $00 \to 10 \to 11$. Once $(x_i,y_i)$ enters the domain 11, the pair $(x_2,y_2)$ is switched on and by the same argument it follows the same path. Repeating this argument, the overall path begins
						\begin{equation}
						\label{1halfcycle}
							000\cdots 00 \rightarrow 100\cdots 00 \rightarrow 110\cdots 00 \rightarrow \dots \rightarrow 111\cdots 11.
						\end{equation}
					Once the pair $(x_{n},y_n)$ enters the domain 11, the pair $(x_1,y_1)$ switches off, and tends towards the domain 00 via the path $11 \to 01 \to 00$. This occurs for each $i$ in turn, so the path continues
						\begin{equation}
						\label{2halfcycle}
					        111\cdots 11 \rightarrow 011\cdots 11 \rightarrow 001\cdots 11 \rightarrow \dots \rightarrow 000\cdots 00.
						\end{equation}
					Concatenating~\eqref{1halfcycle} and~\eqref{2halfcycle} yields the whole cycle. Since each edge of our graph has a unique direction by Lemma~\ref{lem;nobranch} and no edges along the cycle point away from the cycle (no branching), each adjacent edge must point inwards towards the cycle. This makes the cycle qualitatively stable.
					
					Conversely, suppose that the focal point of $(x_i,y_i)$ remains in 00 (resp. 11) for at least one $i$. Then, for such pairs, regardless of where they start, they tend towards the pseudo-state domain $\cdots 00 \cdots$ (resp. $\cdots 11 \cdots$). Once there, they never leave and so no more switching can occur for the $(i+1)^{\textup{th}}$ variable, which prevents switching for the $(i+2)^{\textup{th}}$ variable and so on. Eventually no more switching will occur and therefore no cycle can exist. 
				\end{proof}
		\subsection{Periodic Solutions}\label{subsec:period}
			Although the previous proposition establishes the existence of qualitative cycles, it is natural to ask when a periodic orbit corresponding to such a qualitative cycle actually exists. After all, in principle, it is possible for damped oscillations to occur in a qualitatively stable sequence of domains, converging to a stationary point (a singular stationary point in the limit $q\to 0$). In this section we give a proposition that answers the question of existence of periodic orbits in the smooth case ($q>0$). We start with a necessary lemma. 
				\begin{newlem}
					There exists a unique fixed point for any monotone negative feedback loop in the form of~(\ref{eqn;cfs,chng}).
				\label{fixedpointlem}
				\end{newlem}
				\begin{proof}
					The system (\ref{eqn;cfs,chng}) has a fixed point if and only if 
						\begin{equation}
							\dfrac{b_{1}}{a_{1}}w_{1}=f_{1}\left(\dfrac{c_{2n}a_{2n-1}}{d_{2n}b_{2n-1}}f_{2n-1}\left(\cdots\dfrac{c_{i+1}a_{i}}{d_{i+1}b_{i}}f_{i}\left(\cdots\dfrac{c_{4}a_{3}}{d_{4}b_{3}}f_{3}\left(\dfrac{c_{2}}{d_{2}}w_{1}\right)\right)\right)\right)
						\label{suffcondish}
						\end{equation}
					has a solution. Since $f_{1}$ is a non-negative, monotone decreasing  function, $f_{1}(0)>0$. By these same properties, $\lim_{w_{1} \rightarrow \infty} f_{1}(w_{1})\geq 0$ and is finite. Since $\frac{b_{1}}{a_{1}}w_{1}$ is a line through the origin with positive slope, it follows that (\ref{suffcondish}) has exactly one solution.
				\end{proof}
	
			We answer the question of periodic orbits with the following argument. In a monotone negative feedback loop~(\ref{eqn;cfs,chng}) it is clear that $\frac{\partial \dot{w}_{i}}{\partial w_{i-1}}>0$, $\frac{\partial \dot{w}_{i}}{\partial w_{i}}<0$ is true for each $i=2,\dots,2n$ and that $\frac{\partial \dot{w}_{1}}{\partial w_{2n}}<0$. It is also clear that $\dot{w}_{i}(0,0)\geq 0$ and $\dot{w}_{1}(w_{2n},0)\geq 0$. Furthermore, letting $w_1^*$ be the fixed point of Equation~(\ref{suffcondish}), and $w_i^*$ be the corresponding fixed point for $w_i$ ($w_2^*=\frac{c_2}{d_2}w_1^*$, etc.), we have that for $w_{2n}>w_{2n}^{*}$ and $w_{1}>w_{1}^{*}$
				\begin{displaymath}
					\begin{array}{lcl}
						\dot{w}_{1}& = & -b_{1}w_{1}+ a_{1}f_{1}(w_{2n}) \\
						 & < & -b_{1}w_{1}^{*}+ a_{1}f_{1}(w_{2n}^{*})=0,
					\end{array}
				\end{displaymath}
			and if $w_{2n}<w_{2n}^{*}$ and $w_{1}<w_{1}^{*}$ then 
				\begin{displaymath}
					\begin{array}{lcl}
						\dot{w}_{1}& = & -b_{1}w_{1}+ a_{1}f_{1}(w_{2n}) \\
						 & > & -b_{1}w_{1}^{*}+ a_{1}f_{1}(w_{2n}^{*})=0.
					\end{array}
			\end{displaymath}
			Observing that $\frac{\partial \dot{w}_{1}}{\partial w_{1}}=-b_{1}$, and that the fixed point $w^*=(w_1^*,\ldots,w_{2n}^*)$ has strictly positive coordinates, we can apply Theorem 1 in~\cite{Hastings1977} to establish the following proposition.
				\begin{newthm}
					If the Jacobian, $J(w^{*})$ of the monotone feedback system~(\ref{eqn;cfs,chng}), evaluated at the unique fixed point, $w^*$, has at least one eigenvalue with positive real part and no repeated eigenvalues, then (\ref{eqn;cfs,chng}) has a non-constant periodic solution.
				\label{evenperiodprop} 
				\end{newthm}
				
It is clear that when the assumption in Proposition~\ref{prop: cycleprop} fails, {\em i.e.}, when $0<\theta_i<\frac{c_i}{d_i}\frac{a_i}{b_i}$ fails for some $i$, then there is no qualitative cycle for small enough $q$ (steep enough sigmoids), and the fixed point, $w^*$, lies in its own regular domain and is therefore stable. In such a case, there can be no periodic orbit. When Proposition~\ref{prop: cycleprop} gives a qualitative cycle, then Proposition~\ref{evenperiodprop} determines whether there is a periodic orbit, or only a damped oscillation. 

It is also straightforward to apply in the current case existing theorems for stability as well as existence of periodic orbits under appropriate conditions, mainly on eigenvalues at the fixed point, such as Theorem 4.3 of Mallet-Paret and Smith~\cite{MPS1990}.
				
				\begin{exmp}
					Consider the 2-gene monotone negative feedback loop 
						\begin{equation}
							\begin{array}{l}
								\dot{x}_{1}=1+\frac{3}{2}(1-Z_{2})-\frac{7}{8}x_{1} \\
								\dot{y}_{1}=x_{1}-y_{1} \\
								\dot{x}_{2}=1+\frac{3}{2}Z_{1}-\frac{7}{8}x_{2} \\
								\dot{y}_{2}=x_{2}-2y_{2},
							\end{array}
						\label{cyclicfeedex}
						\end{equation}
					where $\theta_{1}=2$, $\theta_{2}=1$, and $Z_{i}=H(y_{i},\theta_{i},q=0.01)$. The state transition diagram for (\ref{cyclicfeedex}) is given in Figure~(\ref{fig:cycle1}). The only cycle that solutions can enter is the cycle described in Proposition~\ref{prop: cycleprop}. 
						\begin{figure}[!htb]
							\centering
							\includegraphics[scale=.45]{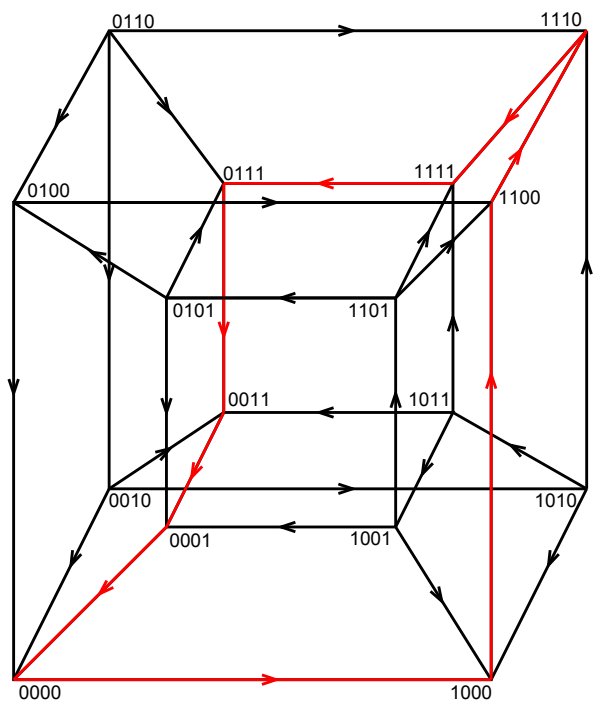}
							\caption{Pseudo-state transition diagram for (\ref{cyclicfeedex}) with the unique cycle in red}
						\label{fig:cycle1}
						\end{figure}
					We note further that there is a fixed point at $(2,2,2,1)$. The Jacobian of (\ref{cyclicfeedex}) is 
						\begin{equation}
						J(2,2,2,1)=\left(
							\begin{array}{cccc}
								-\frac{7}{8} & 0 & 0 & -\frac{75}{2} \\
								1 & -1 & 0 & 0 \\
								0 & \frac{75}{2} & -\frac{7}{8} & 0 \\
								0 & 0 & 1 & -2
							\end{array} \right),
						\label{matrix;(2212)}
						\end{equation}
					which has eigenvalues $\lambda_{1,2}=-5.53045 \pm 4.3162i$ and $\lambda_{3,4}=3.15545 \pm 4.31828i$. Since we have at least one eigenvalue with positive real part and no repeated eigenvalues, Proposition~\ref{evenperiodprop} guarantees that there exists a non-constant periodic orbit. Moreover, by inspection of the PTD, this orbit must be contained in the cycle. 
				\label{cycleex}
Note that
			the two gene system given by (\ref{eqn;instblex}) with $\theta_{11}=1$, $\theta_{12}=2$, $\theta_{21}=\frac{1}{2}$, and $\theta_{22}=1$. From Example~\ref{prop2ex}, we have calculated $\Gamma^{(0)}_{u_{1}}$ and $\Gamma^{(0)}_{u_{2}}$, from which it follows that
				\begin{equation*}
					\begin{array}{l}
						\mathcal{R}^{(0)}_{c_{1}}=\{(y_{1},x_{1})\in\mathbb{R}_{+}^{2}|0<y<1,\ 0<x<\Gamma^{(0)}_{u_{1}}\} \\
						\mathcal{R}^{(0)}_{c_{2}}=\{(y_{2},x_{2})\in\mathbb{R}_{+}^{2}|0<y<\frac{1}{2},\ 0<x<\Gamma^{(0)}_{u_{2}}\},
					\end{array}
				\end{equation*}
			so that $\mathcal{R}^{(00)}_{c}=\{(y_{1},x_{1},y_{2},x_{2})\in\mathbb{R}_{+}^{4}|0<y<1,\ 0<x<\Gamma^{(0)}_{u_{1}}, \ 0<y<\frac{1}{2},\ 0<x<\Gamma^{(0)}_{u_{2}}\}$. 
			By Proposition~\ref{fixpointprop}, if we take an initial condition $(x^{(00)}_{1},y^{(00)}_{1},x^{(00)}_{2},y^{(00)}_{2})\in\mathcal{R}^{(00)}_{c}$, the corresponding trajectory will converge to the fixed point $(0,0,0,0)$. 
			
			However, if we take $x_{1}^{(00)}>\Gamma^{(0)}_{u_{1}}$, then by Proposition~\ref{fixpointprop}, $y_{1}$ will cross $\theta_{11}=1$, at which point the trajectory enters $\mathcal{D}_{10}$. Once $y_{1}$ crosses it's threshold, the pair $(x_{1},y_{1})$ flows asymptotically towards $(\frac{8}{7}, \frac{8}{7})$. Simultaneously, the pair $(x_{2},y_{2})$ flows towards $(\frac{8}{7},\frac{4}{7})$. By the structure of the flow, once $y_{1}>1$, the region $\mathcal{R}:=\cup_{i\in K}\mathcal{P}_{i}$ where $K$ is the set of all permutations of strings of length four with elements 1 or 2, is invariant. Thus, once $y_{2}$ crosses $\theta_{21}=\frac{1}{2}$, the equations permanently 
take the form of~(\ref{cyclicfeedex}).
			
			Once $y_{2}$ enters $\mathcal{R}$, (\ref{eqn;instblex}) enters the cycle described in Proposition~\ref{prop: cycleprop}. Since $(2,2,2,1)$ is still a fixed point, (\ref{matrix;(2212)}) is the Jacobian at $(2,2,2,1)$. From Example~\ref{cycleex}, this has two eigenvalues with positive real part, so (\ref{eqn;instblex}) has at least one non-constant periodic solution, by Proposition~\ref{evenperiodprop}. 
			
			This example illustrates how the long term dynamics or the system are dependent on their initial conditions with respect to the curves $\Gamma_{u_{i}}^{(h_{i})}$ or $\Gamma_{l_{i}}^{(h_{i})}$. 
			\qed
	
\end{exmp}

	\section{Discussion}

This study has shown how it is possible to reproduce or extend results on Glass networks and corresponding steep sigmoidal networks (Model 1) to the context of transcription-translation networks (Model 2), and to clarify what differences arise.
The pseudo-state transition diagram takes the place of the state transition diagram for Glass networks, and re-establishes the property that edges can only be traversed in one direction. For transcription-translation networks, fixed points in regular domains (or even pseudo-state domains) are not globally attracting within their domain. However, black and white walls are avoided, and all walls are transparent. This can lead to rapid oscillations across walls. Interestingly, nonuniqueness can still arise in the step-function limit, where trajectories graze walls at pseudo-state boundary intersections. Integrating trajectories from switching point to switching point can still be done for transcription-translation networks in the step-function limit, but this typically involves solutions to transcendental equations that must be computed numerically. A result on periodic behaviour in negative feedback loops has been proven, similar to results for Glass networks. This is an important example, but it remains for future work to carry out the analysis of more complex structures for this class of transcription-translation networks, as has been done for Glass networks. We anticipate, for example, that results for existence of periodic solutions or multistability in more general network structures (within which negative or positive feedback loops occur) could be established for our class of transcription-translation models, similar to such results for Glass networks (e.g. ~\cite{Snoussi1989}) and other classes of networks (e.g.~\cite{Comet2013}).

	\section*{Acknowledgments}
		The authors would like to thank Dr. Kieka Mynhardt for helpful discussions concerning Section~\ref{subsec:cancel}, and the Natural Sciences and Engineering Research Council of Canada for funding this research. 
	
	\section*{References}

		 \newpage




\begin{thebibliography}{99}
	
	
	
	\bibitem{Bernstein2002} Bernstein JA, Khodursky AB, Lin PH, Lin-Chao S, Cohen SN (2002) Global analysis of mRNA decay and abundance in {\it Escherichia coli} at single-gene resolution using two-color fluorescent DNA microarrays. PNAS 99:9697--9702. 
	
	\bibitem{Bionumbers} Bionumbers: http://bionumbers.hms.harvard.edu (accessed 9 Nov. 2014).
	

	\bibitem{Comet2013} Comet J-P, Noual M, Richard A, Aracena J, Calzone L, Demongeot J, Kaufman M, Naldi A, Snoussi, EH, Thieffry D (2013) On circuit functionality in Boolean networks. Bull Math Biol 75:906--919.
	
	\bibitem{Edwards2000} Edwards R (2000) Analysis of continuous-time switching networks. Physica D 146:165--199.
	
	\bibitem{eff2012} Edwards R, Farcot E, Foxall E (2012) Explicit construction of chaotic attractors in Glass networks. Chaos, Solitons and Fractals 45:666--680.
	
	\bibitem{EI2013} Edwards R, Ironi, L (2014) Periodic solutions of gene networks with steep sigmoidal regulatory functions. Physica D (accepted).
	
	\bibitem{Edwards2014} Edwards R, Machina A, McGregor G, van den Driessche P (2014) A modelling framework for gene regulatory networks including transcription and translation.  Bull. Math. Biol. (under review).
	
	
	\bibitem{Farcot2006} Farcot E (2006) Geometric properties of a class of piecewise affine biological network models. J Math Biol 52:373--418.
	
	\bibitem{FG2009} Farcot E, Gouz\'e J-L (2009) Periodic solutions of piecewise affine gene network models with non uniform decay rates: The case of a negative feedback loop. Acta Biotheor 57:429--455.
	
	
	
	
	
	\bibitem{Gedeon2012} Gedeon T, Cummins G, Heys JJ (2012) Effect of model selection on prediction of periodic behavior in gene regulatory networks. Bull Math Biol 74:1706--1726.
	
	\bibitem{GK1973} Glass L, Kauffman S (1973) The logical analysis of continuous non-linear biochemical control networks. J Theor Biol 39:103--129.
	
	\bibitem{Glass1975} Glass L. (1975) Combinatorial and topological methods in nonlinear chemical kinetics. J Chem Phys 63:1325--1335.
	
	\bibitem{Glass1977} Glass L. (1977) Global Analysis of Nonlinear Chemical Kinetics, in: Statistical Mechanics, Part B: Time-Dependent Processes, B. J. Berne, ed. (Plenum, New York, 1977) pp.311--349.
	
	\bibitem{gp78b} Glass L, Pasternack JS (1978) Stable oscillations in mathematical models of biological control systems, J Math Biol 6:207--223.
	
	\bibitem{gs2002} Gouz\'e J-L, Sari T (2002) A class of piecewise linear differential equations arising in biological models. Dyn Syst 17:299-316.
	
	\bibitem{Hastings1977} Hastings S, Tyson J, Webster D (1977) Existence of Periodic Solutions for Negative Feedback Cellular Control Systems. Journal of Differential Equations 25, 39 - 64.
	
	
	\bibitem{ek2005} Killough DB, Edwards R (2005) Bifurcations in Glass networks. Int J Bif Chaos 15:395--423.
	
	
	symbolic dynamics of neural networks. Neural Computation 4:621--642.
	
	
	
	
	\bibitem{MPS1990} Mallet-Paret J, Smith H (1990) The Poincare-Bendixson theorem for monotone cyclic feedback systems. Journal of Dynamics and Differential Equations 1572-9222.
	
	\bibitem{Mosteller1980} Mosteller RD, Goldstein RV, Nishimoto KR (1980) Metabolism of individual proteins in exponentially growing {\it Escherichia coli}. J. Biol. Chem. 255:2524--2532.
	
	\bibitem{Paetkau2006} Paetkau V, Edwards R, Illner R (2006) A model for generating circadian rhythm by coupling ultradian oscillators. Theoretical Biology and Medical Modelling 3:12, pp.1-10.
	
	\bibitem{pk2005} Plahte E, Kj{\o}glum S (2005) Analysis and generic
	properties of gene regulatory networks with graded response
	functions. Physica D 201:150--176.
	
	ODE modelling approaches for gene regulatory networks. Journal of Theoretical Biology 261(4): 511–-530.
	
	
	
	
	\bibitem{Snoussi1989} Snoussi H (1989) Qualitative dynamics of piecewise-linear differential equations: a discrete mapping approach. Dyn Stab Syst 4:189--207.
	
	
	
	
	
	\bibitem{Wittmann2009} Wittmann, DM, {\it et al.} (2009) Transforming Boolean models to continuous models: methodology and application to T-cell receptor signaling. BMC Systems Biology 3:98.
	
	\end{thebibliography}





\end{document}